\documentclass[smallextended,envcountsect]{svjour3}
\smartqed
\usepackage{graphicx,epsfig}
\usepackage{color,amsmath,amssymb,tikz,mathrsfs}
\usepackage{todonotes}
\usepackage{geometry}
\journalname{JOTA}

\newtheorem{algo}{Algorithm}
\numberwithin{equation}{section}

\begin{document}

\title{Shape Derivative for Penalty-Constrained Nonsmooth--Nonconvex Optimization: Cohesive Crack Problem}

\titlerunning{Shape Derivative for Penalty-Constrained Optimization: Cohesive Crack Problem}
% if too long for running head

\author{Victor A. Kovtunenko \and Karl Kunisch}

\institute{V.A. Kovtunenko \at
Institute for Mathematics and Scientific Computing, Karl-Franzens University of Graz,
NAWI Graz, Heinrichstra\ss{}e 36, 8010 Graz, Austria
\email{victor.kovtunenko@uni-graz.at} \\
\emph{and}
Lavrentyev Institute of Hydrodynamics, Siberian Division of the Russian Academy of Sciences,
630090 Novosibirsk, Russia
\and
K. Kunisch \at
Institute for Mathematics and Scientific Computing, Karl-Franzens University of Graz,
NAWI Graz, Heinrichstra\ss{}e 36, 8010 Graz, Austria
\email{karl.kunisch@uni-graz.at} \\
\emph{and}
Radon Institute, Austrian Academy of Sciences,
RICAM Linz, Altenbergerstra\ss{}e 69, 4040 Linz, Austria
}

\date{}
%The correct dates will be entered by the editor.

\maketitle

\begin{abstract}
A class of non-smooth and non-convex optimization problems with penalty constraints
linked to  variational inequalities (VI) is studied with respect to its shape differentiability.
The specific problem stemming from quasi-brittle fracture describes
an elastic body with a Barenblatt cohesive crack under the inequality condition
of non-penetration at the crack faces.
Based on the Lagrange approach and using smooth penalization with the Lavrentiev
regularization, a formula for the shape derivative is derived.
The explicit formula contains both primal and adjoint states and is useful
for finding descent directions for a gradient algorithm
to identify an optimal crack shape from a boundary measurement.
Numerical examples of destructive testing are presented in 2D.
\end{abstract}
\keywords{Shape optimization \and Optimal control \and Variational inequality
\and Penalization \and Lagrange method \and Lavrentiev regularization \and
Free discontinuity problem \and Non-penetrating crack \and Quasi-brittle fracture
\and Destructive physical analysis}
\subclass{35R37 \and 49J40 \and 49Q10 \and 74RXX}

\section{Introduction}\label{sec0}

We develop a shape derivative of geometry-dependent least-squares
functions for a class of non-smooth and non-convex optimization problems.
The shape optimization problem is constrained by a penalty equation
linked to a variational inequality (VI).
The specific problem describes non-penetrating cracks with cohesion
in the framework of quasi-brittle fracture and destructive physical analysis (DPA).

Within the general theory for optimal control of VI \cite{Bar/84,MP/84},
the main challenge consists in the derivation of optimality conditions.
It can be studied by proper approximation of VI by regularized equations
and taking the limit as the regularization parameter tends to zero.
The corresponding methods for optimal control of obstacle problems
can be found in \cite{Ber/97,HHL/14} using augmented Lagrangians, 
in e.g. \cite{HK/09,IK/03} for a Moreau--Yosida regularization, and in \cite{NT/09}
based on a Lavrentiev regularization, for the latter see \cite{HKR/16,Lav/67}.
Furthermore we cite \cite{BLR/15,CCK/13} for control of non-smooth and non-convex
functionals, \cite{KV/07} for boundary control, and \cite{GJKS/18,ZMK/21}
for control of quasi- and hemi-VI.
Shape optimization for free-interface identification with obstacle-type VI
using adjoints was developed recently by \cite{FSW/18,LSW/20}.
The common difficulty is a lack of regularity that needs
assumptions on a solution in order to take the limit \cite{Sch/22}.

Relying on linearized relations, a crack identification problem was treated e.g. in \cite{BB/13}.
We can refer also to \cite{ATS/21,HIIS/19} for relevant shape optimization
problems in acoustics, to \cite{GGK/21} in nonlinear flows subject to the divergence-free constraint,
to \cite{KO/20} for over-determined and to \cite{HKKP/03} for Bernoulli-type free boundary problems.
In the case of non-penetrating cracks (which are inequality-constrained),
the shape differentiability of the bulk energy was proved in
\cite{FHLRS/09,KK/00} for rectilinear cracks
and used for optimal shape design in \cite{KS/18,KMZ/08,LI/19,LSZ/15}.
For curvilinear cracks, adopting the theorem of
Correa--Seeger \cite{CS/85} on directional differentiability of Lagrangians
the shape derivative was derived in \cite{Kov/06,KK/07},
and in \cite{Rud/04} using $\Gamma$-convergence.

For the non-penetrating Barenblatt crack that we investigate here,
the study of the objective function and its optimal control with respect to the crack shape
has a number of challenging tasks that we address below.
The subsequent Sections~\ref{sec2}--\ref{sec6} follow Tasks (i)--(v),
which for convenience are summarized  and explained in the following Sections~\ref{sec1}.

\section{Modeling tasks}\label{sec1}

Let $t\mapsto\Omega_t$ by a parameter (time)-dependent geometry with
a crack $\Gamma_t$ along an interface (the breaking line) $\Sigma_t$.
Denote by $\nu_t$ a normal vector to the surface $\Sigma_t$.
Motivated by applications in fracture mechanics (see e.g. \cite{BMP/09}), we consider
a total energy functional $u\mapsto\mathcal{E}: V(\Omega_t)\mapsto\mathbb{R}$,
which is given in a Hilbert space $V(\Omega_t)$ by the sum
\begin{equation}\label{1.1}
\mathcal{E}(u; \Omega_t) =\mathcal{B}(u; \Omega_t) +\mathcal{S}([\![u]\!];\Sigma_t),
\end{equation}
where the bulk term $\mathcal{B}$ is convex, typically, quadratic.
The term $\mathcal{S}$ describes a surface energy according to
the Barenblatt idea of a cohesion zone and depends on the jump $[\![u]\!]$
expressing a possible discontinuity across the interface $\Sigma_t$ field $u$.
The latter term is non-convex.
The condition of non-penetration (see \cite{KK/00,KS/97}) for the normal opening
$\nu_t\cdot[\![u]\!]\ge0$ describes the feasible set $K(\Omega_t)\subset V(\Omega_t)$
which is a convex cone.
For differentiable maps $u\mapsto\mathcal{E}$, the first order optimality condition
for the minimization of $\mathcal{E}(u; \Omega_t)$ over $u\in K(\Omega_t)$
results in a VI
\begin{equation}\label{1.2}
u_t\in K(\Omega_t),\quad
\langle \partial_u \mathcal{E}(u_t; \Omega_t), u -u_t \rangle\ge0
\quad\text{for all } u\in K(\Omega_t).
\end{equation}
It constitutes a non-convex problem for a solid with a non-penetrating crack (see \cite{Kov/05}).

For comparison, the classic Griffith model of brittle fracture
simplifies $\mathcal{S}$ to be constant, and a crack $\Gamma_t$
to be predefined at the interface $\Sigma_t$.
This simplification results in a square-root singularity of the displacement $u_t$
and infinite stress at the crack tip (front) $\partial\Gamma_t$.
This is the main disadvantage of the Griffith model,
we refer to \cite{CFMT/00} for a discussion.
A model, consistent with the physics of quasi-brittle fracture for non-constant
$\mathcal{S}$, was suggested by Barenblatt \cite{Bar/62}.
It takes into account the surface cohesion from the meso-level such that
the interface surfaces close in a smooth way, and thus allow healing of the crack.
Indeed, after solving problem \eqref{1.2} according to Barenblatt,
the set of points where an opening $[\![u_t]\!]\not=0$ occurs, determines
the a-priori unknown crack $\Gamma_t$ along the interface $\Sigma_t$.
This is the complement to the closed part of the interface where $[\![u_t]\!]=0$.

The main challenge of the direct problem \eqref{1.2} concerns
the term $\mathcal{S}$ in \eqref{1.1}.
From an optimization point of view, minimization over feasible $u\in K(\Omega_t)$
of $\mathcal{E}$ with a non-smooth surface density
$[\![u]\!]\mapsto\mathcal{S}$ (when not a $C^1$-function) leads to a hemi-VI \eqref{1.2}.
The hemi-VI approach was analyzed theoretically and numerically
in \cite{HKK/11,OG/14} and used in \cite{Kov/11,KS/06,LPSS/13}
to describe a quasi-static crack propagation.
A quadratic function $\mathcal{S}$ describing adhesive cracks was studied in \cite{FIR/20}.
In the present paper, we study $C^2$-smooth surface energies $\mathcal{S}$
that are small compared to the bulk term $\mathcal{B}$ in \eqref{1.1},
see assumption \eqref{3.14} below, which is consistent with meso-level modeling.

Our ultimate aim is to identify the free-interface $\Sigma_t$
by a shape optimization approach as described in \cite{GKK/20}.
For this task, we introduce the VI-constrained least-squares misfit from
a given measurement $z$ at an observation boundary $\Gamma^{\rm O}_t$:
\begin{equation}\label{1.3}
\mathcal{J}(u_t; \Omega_t)
=\frac{1}{2} \int_{\Gamma^{\rm O}_t} |u_t - z|^2 \,dS_x +\rho |\Sigma_t|
\quad\text{such that $u_t$ solves \eqref{1.2}},
\end{equation}
where the regularization uses parameter $\rho>0$.
This constitutes a nonsmooth--nonconvex optimization problem.

Our current work focuses on the following tasks.

\paragraph{Task (i): $C^2$-approximation of $\mathcal{E}$.}
To provide a shape derivative of $\mathcal{J}$ defined in \eqref{1.3}
a continuously differentiable approximation of VI \eqref{1.2} is needed.
The standard penalization of non-penetration  $\nu_t\cdot[\![u]\!]\ge0$
by $-[\nu_t\cdot[\![u]\!] ]^-/\varepsilon$ has only $C^0$-regularity.
Here the regularization parameter $\varepsilon>0$ is small, and
$u =[u]^+ -[u]^-$ implies the decomposition into positive
$[u]^+ =\max(0,u)$ and negative $[u]^- =-\min(0,u)$ parts.
Therefore, we suggest a $C^1$-penalization by the normal compliance
$\beta_\epsilon(\nu_t\cdot[\![u]\!])$ based on the Lavrentiev regularization
(see Theorem~\ref{theo2}).
This results in a $C^2$-approximation of $\mathcal{E}$ for
the $\varepsilon$-approximation of \eqref{1.1}--\eqref{1.3} by
\begin{equation}\label{1.4}
\mathcal{J}(u^\varepsilon_t; \Omega_t)
=\frac{1}{2} \int_{\Gamma^{\rm O}_t} |u^\varepsilon_t - z|^2 \,dS_x +\rho |\Sigma_t|,
\quad\text{where }\partial^\varepsilon_u \mathcal{E}(u^\varepsilon_t; \Omega_t)=0,
\end{equation}
and the penalty equation involves the operator $\partial^\varepsilon_u \mathcal{E}$ introduced as  follows
\begin{equation}\label{1.5}
\langle \partial^\varepsilon_u \mathcal{E}(u; \Omega_t), v \rangle
:=\langle \partial_u \mathcal{E}(u; \Omega_t), v \rangle
+\int_{\Sigma_t} \beta_\epsilon(\nu_t\cdot[\![u]\!])\, (\nu_t\cdot[\![v]\!]) \,dS_x.
\end{equation}

\paragraph{Task (ii): adjoint-based optimality conditions.}
Applying to the penalty-constrained least-square misfit \eqref{1.4} a Lagrange
multiplier approach (see \cite{IK/08}), we can define an $\varepsilon$-dependent Lagrangian
$(u,v)\mapsto \mathcal{L}^\varepsilon: V(\Omega_t)^2\mapsto\mathbb{R}$ as
\begin{equation}\label{1.6}
\mathcal{L}^\varepsilon(u, v; \Omega_t) =\mathcal{J}(u; \Omega_t)
-\langle \partial^\varepsilon_u \mathcal{E}(u; \Omega_t), v \rangle.
\end{equation}
The primal (inf-sup) problem: for fixed $v^\varepsilon_t\in V(\Omega_t)$,
find $u^\varepsilon_t\in V(\Omega_t)$ such that
\begin{equation}\label{1.7}
\mathcal{L}^\varepsilon(u^\varepsilon_t, v;\Omega_t)
\le \mathcal{L}^\varepsilon(u^\varepsilon_t, v^\varepsilon_t;\Omega_t)
\quad\text{for all } v\in V(\Omega_t).
\end{equation}
Since $\mathcal{L}^\varepsilon$ is affine in $v$,
the first order optimality condition is given by
\begin{equation*}
u^\varepsilon_t\in V(\Omega_t),\quad
\langle \partial^\varepsilon_u \mathcal{E}(u^\varepsilon_t; \Omega_t), u \rangle =0
\quad\text{for all } u\in V(\Omega_t).
\end{equation*}
The dual (sup-inf) problem (see \cite[Chapter~6]{ET/76}) reads:
for fixed $u^\varepsilon_t\in V(\Omega_t)$,
find $v^\varepsilon_t\in V(\Omega_t)$ such that
\begin{equation*}
\mathcal{L}^\varepsilon(u^\varepsilon_t, v^\varepsilon_t; \Omega_t)
\le \mathcal{L}^\varepsilon(u, v^\varepsilon_t; \Omega_t)
\quad\text{for all } u\in V(\Omega_t) .
\end{equation*}
Note that  $\mathcal{L}^\varepsilon$ with respect to $u$ is not a
linear continuous functional on the dual space $V(\Omega_t)^\star$.

The corresponding nonlinear optimization theory was developed
in e.g. \cite{IK/08,MZ/79,ZK/79} as follows.
If the variation $\partial_u(\partial^\varepsilon_u \mathcal{E})\in
\mathscr{L}(V(\Omega_t), V(\Omega_t)^\star)$ with respect to $u$ in \eqref{1.5} exists,
and the associated adjoint operator $[\partial_u(\partial^\varepsilon_u \mathcal{E})
(u^\varepsilon_t; \Omega_t)]^\star\in \mathscr{L}(V(\Omega_t), V(\Omega_t)^\star)$
satisfying
\begin{equation*}
\langle [\partial_u(\partial^\varepsilon_u \mathcal{E}) (u^\varepsilon_t; \Omega_t)]^\star
v, u \rangle
=\langle \partial_u(\partial^\varepsilon_u \mathcal{E}) (u^\varepsilon_t; \Omega_t)
u, v \rangle \quad\text{for all } u,v\in V(\Omega_t)
\end{equation*}
is surjective with respect to $u^\varepsilon_t$, then the optimality condition is given by
\begin{equation}\label{1.8}
v^\varepsilon_t\in V(\Omega_t),\quad
\langle [\partial_u(\partial^\varepsilon_u \mathcal{E}) (u^\varepsilon_t; \Omega_t)]^\star
v^\varepsilon_t, v \rangle =0 \quad\text{for all } v\in V(\Omega_t).
\end{equation}
For the abstract theory  associated to adjoint operators we cite \cite{EL/13,Ibr/06,MAS/96}.
To justify \eqref{1.8}, we shall linearize $\partial_u \mathcal{E}$ around the primal solution
$u^\varepsilon_t$ to \eqref{1.7} (see Theorem~\ref{theo3}) and suggest a suitable
linearized functional $(u,v)\mapsto\tilde{\mathcal{L}}^\varepsilon (0, u^\varepsilon_t, u, v):
V(\Omega_t)^2\mapsto\mathbb{R}$ such that
\begin{equation}\label{1.9}
\tilde{\mathcal{L}}^\varepsilon (0, u^\varepsilon_t, u^\varepsilon_t, v; \Omega_t)
=\mathcal{L}^\varepsilon(u^\varepsilon_t, v; \Omega_t)
\quad\text{for } v\in V(\Omega_t) .
\end{equation}

\paragraph{Task (iii): shape derivative.}
Our purpose is to calculate a shape derivative of the mapping
$t\mapsto\mathcal{J}(u^\varepsilon_t; \Omega_t)$
that is expressed by the one-sided limit (see \cite{DZ/11,SZ/92}):
\begin{equation}\label{1.10}
\partial_t \mathcal{J}(u^\varepsilon_t; \Omega_t)
=\lim_{s\to0^+} \frac{1}{s}
\bigl( \mathcal{J}(u^\varepsilon_{t+s}; \Omega_{t+s})
-\mathcal{J}(u^\varepsilon_t; \Omega_t) \bigr).
\end{equation}
If a saddle-point $(u^\varepsilon_t, v^\varepsilon_t)\in V(\Omega_t)^2$
based on \eqref{1.7} and \eqref{1.9} exists, then the optimal value misfit
function defined in \eqref{1.4} is evidently equal to the optimal value Lagrange function
\begin{multline}\label{1.11}
\tilde{\mathcal{L}}^\varepsilon (0, u^\varepsilon_t, u^\varepsilon_t, v^\varepsilon_t; \Omega_t)
=\mathcal{L}^\varepsilon(u^\varepsilon_t, v^\varepsilon_t; \Omega_t)
=\mathcal{J}(u^\varepsilon_t; \Omega_t)
-\langle \partial^\varepsilon_u \mathcal{E}(u^\varepsilon_t; \Omega_t),
v^\varepsilon_t \rangle \quad\text{subject to }\\
\tilde{\mathcal{L}}^\varepsilon (0, u^\varepsilon_t, u^\varepsilon_t, v;\Omega_t) \le
\tilde{\mathcal{L}}^\varepsilon (0, u^\varepsilon_t, u^\varepsilon_t, v^\varepsilon_t;\Omega_t)
\le \tilde{\mathcal{L}}^\varepsilon (0, u^\varepsilon_t, u, v^\varepsilon_t;\Omega_t)\\
\text{for all } (u, v)\in V(\Omega_t)^2.
\end{multline}
Henceforth, we have the following identity for the shape derivative according to \eqref{1.10}:
\begin{multline}\label{1.12}
\partial_t \mathcal{J}(u^\varepsilon_t; \Omega_t) =\partial_t \tilde{\mathcal{L}}^\varepsilon
(0, u^\varepsilon_t, u^\varepsilon_t, v^\varepsilon_t; \Omega_t)\\
=\lim_{s\to0^+} \frac{1}{s} \bigl( \tilde{\mathcal{L}}^\varepsilon
(s, u^\varepsilon_t, u^\varepsilon_{t+s}, v^\varepsilon_{t+s}; \Omega_{t})
-\tilde{\mathcal{L}}^\varepsilon (0, u^\varepsilon_t, u^\varepsilon_t, v^\varepsilon_t;
\Omega_t) \bigr).
\end{multline}
In order to construct a proper $\tilde{\mathcal{L}}^\varepsilon$,
using a diffeomorphic coordinate transformation $y=\phi_s(x)$ such that
$\phi_s: \Omega_t\mapsto\Omega_{t+s}$ (see \cite[Chapter~2]{SZ/92}),
the bijection $ V(\Omega_{t+s})\mapsto  V(\Omega_{t})$, $u\mapsto u\circ\phi_s$
provides the perturbed Lagrangian as
\begin{equation}\label{1.13}
\widetilde{\mathcal{L}}^\varepsilon(s, u\circ\phi_s, u\circ\phi_s, v\circ\phi_s; \Omega_t)
=\mathcal{L}^\varepsilon(u, v; \Omega_{t+s}) \quad\text{for } (u, v)\in V(\Omega_{t+s})^2.
\end{equation}
Then the results of Delfour--Zolesio \cite{DZ/11} on shape differentiabiliy
can be applied to justify the limit in \eqref{1.12},
see respective Theorem~\ref{theo4} and its Corollary~\ref{corol1}.

\paragraph{Task (iv): limit as $\varepsilon\to0^+$.}
Taking the limit as $\varepsilon\to0^+$ in relations \eqref{1.11}
we shall prove the  optimality conditions
(see Theorem~\ref{theo5} and its Corollary~\ref{corol3}).
However, we cannot pass to the limit in \eqref{1.12} due to the presence
of the unbounded term $\beta_\varepsilon^\prime$.
We conjecture that the limit problem \eqref{1.2} is not differentiable.
This agrees with the assertion that VIs are not Fr\'{e}chet differentiable
with respect to shape (see \cite{LSZ/15}).
Therefore, in the numerical treatment we rely on the approximation \eqref{1.12}
with small $\varepsilon>0$ for the shape derivative $\partial_t \mathcal{J}$.

\paragraph{Task (v): shape optimization.}
Commonly adopted in shape optimization, the gradient method needs
a descent direction minimizing the objective map $\Omega_t\mapsto \mathcal{J}$
such that $\partial_t \mathcal{J}<0$.
This can be attained by a proper choice of the transformation $\phi_s$ entering
implicitly in formula \eqref{1.12} (see Corollary~\ref{corol2}).
Realizing the optimization algorithm for crack shape identification,
from our numerical tests we report the following feature.
Those parts of the crack faces which are in contact
(where the non-penetration constraint is active) are hidden from identification.
To identify a crack needs its faces to be open (that is, VI turns into unconstrained equation)
in accordance with the concept of destructive physical analysis (DPA).

\section{Cohesive crack problem}\label{sec2}

We start with a detailed description of the geometry.
Let $\Omega\subset\mathbb{R}^d$, $d=2,3$, be a fixed hold-all domain
with Lipschitz boundary $\partial\Omega$.
For the time-parameter $t\in(t_0,t_1)$,  $t_0<t_1$, we consider a parameter-dependent
geometry $\Omega_t =(\Gamma^{\rm D}_t, \Gamma^{\rm N}_t,
\Gamma^{\rm O}_t, \Sigma_t)$ defined as follows. 
For brevity we use a single notation $\Omega_t$ for the collection of 
geometric objects describing a broken domain $\Omega\setminus\Sigma_t$ 
by means of the Dirichlet, Neumann, observation boundaries, 
and the breaking line, respectively.

The outer boundary is split into two variable parts such that
$\partial\Omega =\overline{\Gamma^{\rm D}_t}\cup \overline{\Gamma^{\rm N}_t}$
and $\Gamma^{\rm D}_t\cap \Gamma^{\rm N}_t =\emptyset$
with normal vector $n_t$ outward to $\Omega$.
The observation boundary is $\Gamma^{\rm O}_t\subset \Gamma^{\rm N}_t$.
The domain is split into two variable sub-domains $\Omega^\pm_t$
with Lipschitz boundaries $\partial\Omega^\pm_t$ and outward
normal vectors $n^\pm_t$ such that $n^\pm_t =n_t$ at $\partial\Omega$.
The conditions $\Gamma^{\rm D}_t\cap\partial\Omega^+_t\not=\emptyset$
and $\Gamma^{\rm D}_t\cap\partial\Omega^-_t\not=\emptyset$
are needed to guarantee the Korn--Poincare inequality.
These two domains are separated by a breaking manifold (the free-interface)
$\Sigma_t =\partial\Omega^+_t\cap \partial\Omega^-_t$
with normal direction $\nu_t =n^-_t =-n^+_t$
such that $\Omega =\Omega^+_t\cup \Omega^-_t\cup \Sigma_t$.
%We assume that $\nu_t$ is differentiable.
An example geometry of $\Omega_t$ is sketched in 2D in Figure~\ref{fig1}.
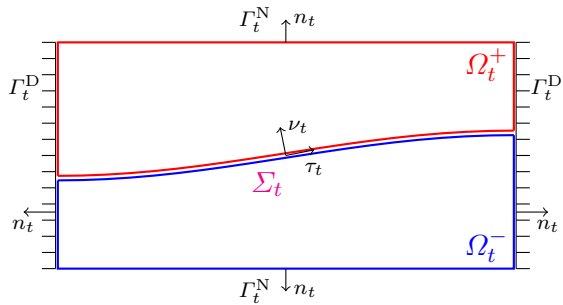
\begin{figure}[hbt!]
\begin{center}
\begin{tikzpicture}[scale=1.5]%unit=1cm
\foreach \X in {2} {
\foreach \Y in {1} {
\foreach \d in {\Y/50} {
\foreach \A in {\Y/5} {
\foreach \K in {pi/2} {
\draw [thick, red, domain=-\X:\X, samples=20]
plot (\x, {\A*sin(\K*\x/\X r)+\d});
\draw [thick, red, -] (-\X,{\A*sin(-\K r)+\d})--(-\X,\Y)--(\X,\Y)--(\X,{\A*sin(\K r)+\d});
%Omega^+
\draw [thick, blue, domain=-\X:\X, samples=20]
plot (\x, {\A*sin(\K*\x/\X r)-\d});
\draw [thick, blue, -] (-\X,{\A*sin(-\K r)-\d})--(-\X,-\Y)--(\X,-\Y)--(\X,{\A*sin(\K r)-\d});
%Omega^-
\draw [-] (\X*1.01,-\Y)--(\X*1.01,\Y)
node[pos=.8, right] {$\;\Gamma^{\rm D}_t\;$}
node[red, pos=.9, left] {\large $\Omega^+_t$}
node[blue, pos=.1, left] {\large $\Omega^-_t$};
\draw [->] (\X*1.01,-\Y*.5)--(\X*1.15,-\Y*.5)
node[pos=1, below] {$n_t$};
%right boundary
\draw [-] (-\X*1.01,-\Y)--(-\X*1.01,\Y)
node[pos=.8, left] {$\Gamma^{\rm D}_t\;$};
\draw [->] (-\X*1.01,-\Y*.5)--(-\X*1.15,-\Y*.5)
node[pos=1, below] {$n_t$};
%left boundary
\draw [->] (0,\Y)--(0,\Y*1.2)
node[pos=1, left] {$\Gamma^{\rm N}_t\;$}
node[pos=1, right] {$n_t$};
%upper boundary
\draw [->] (0,-\Y)--(0,-\Y*1.2)
node[pos=1, left] {$\Gamma^{\rm N}_t\;$}
node[pos=1, right] {$n_t$};
%lower boundary
\draw [->] (0,0)--(-.05,.25)
node[magenta, pos=-1, left] {\large $\Sigma_t$}
node[pos=1, right] {$\nu_t$};
\draw [->] (0,0)--(.25,.05)
node[pos=1, below=2pt] {$\tau_t$};
%interface
\foreach \N in {7} {
\foreach \n in {0,...,\N} {
\draw [-] (\X*1.01,\n*\Y/\N)--(\X*1.07,\n*\Y/\N)
(\X*1.01,-\n*\Y/\N)--(\X*1.07,-\n*\Y/\N)
(-\X*1.01,\n*\Y/\N)--(-\X*1.07,\n*\Y/\N)
(-\X*1.01,-\n*\Y/\N)--(-\X*1.07,-\n*\Y/\N);
}}
%clamping
}}}}};
\end{tikzpicture}
\end{center}
\caption{An example configuration of variable geometry $\Omega_t$ in 2D.}\label{fig1}
\end{figure}
We assume that these geometric properties are preserved for all  $t\in(t_0,t_1)$
under suitable shape perturbations, which we specify below in Section~\ref{sec4}.

For fixed $t$, we consider a linear elastic body that occupies the disconnected domain
$\Omega\setminus\Sigma_t =\Omega^+_t\cup \Omega^-_t$.
By this, $d$-dimensional vectors of displacement $u(x)$ at points
$x\in\Omega\setminus\Sigma_t$ admit discontinuity across $\Sigma_t$
resulting in the jump $[\![u]\!] =u|_{\Sigma_t\cap \partial\Omega^+_t}
-u|_{\Sigma_t\cap \partial\Omega^-_t}$.
For further use we employ an orthogonal decomposition of admissible $[\![u]\!]$
into the normal component with factor $\nu_t\cdot[\![u]\!]$ and the tangential vector
$[\![u]\!]_{\tau_t}$ at the interface such that
\begin{equation}\label{2.1}
[\![u]\!] =(\nu_t\cdot[\![u]\!])\nu_t +[\![u]\!]_{\tau_t},
\quad \nu_t\cdot[\![u]\!]\ge0\quad\text{on }\Sigma_t.
\end{equation}
The latter inequality in \eqref{2.1} describes the non-penetration, see \cite{KK/00}.

The essential issue of modeling is to introduce a density at $\Sigma_t$
for the surface energy $\mathcal{S}$ in \eqref{1.1} that is consistent with physics.
Based on the decomposition \eqref{2.1}, we set
\begin{equation}\label{2.2}
\mathcal{S}([\![u]\!];\Sigma_t) =\int_{\Sigma_t} \bigl\{
\alpha_{\rm f}([\![u]\!]_{\tau_t}) +\alpha_{\rm c}(\nu_t\cdot[\![u]\!]) \bigr\} \,dS_x.
\end{equation}
The former, shear-induced term in \eqref{2.2}, is associated with friction
between the crack surfaces.
Let the mapping $\xi\mapsto\alpha_{\rm f}(\xi): \mathbb{R}^d\mapsto\mathbb{R}$,
and its first and second derivatives be uniformly continuous functions, satisfying
for constants $K_{\rm f}, K_{\rm f 1}, K_{\rm f 2}\ge 0$ and all $\xi$,
\begin{equation}\label{2.3}
-K_{\rm f} |\xi|\le\alpha_{\rm f}(\xi),\quad |\nabla \alpha_{\rm f}(\xi)|\le K_{\rm f 1},
\quad |\nabla^2 \alpha_{\rm f}(\xi)|\le K_{\rm f 2}.
\end{equation}
For example, we have in mind a standard regularization of the Coulomb law
(see e.g. \cite[Section~4.3.3]{SM/91}) with the positive, convex function
\begin{equation}\label{2.4}
\alpha_{\rm f}(\xi) =F_{\bf b} \sqrt{\delta^2 +|\xi|^2},
\end{equation}
where $\delta>0$ is small, and $F_{\bf b}>0$ is the friction bound.
In this case, $K_{\rm f} =0$, $K_{\rm f 1} =F_{\bf b}$,
and $K_{\rm f 2} =F_{\bf b}/\delta$.
For convenience, the function $\alpha_{\rm f}(s)$ in one variable $s\in\mathbb{R}$
together with its first two derivatives are depicted in Figure~\ref{fig2}.
\begin{figure}[hbt]
\begin{minipage}[b]{.33\linewidth}
\begin{center}
\raisebox{0mm}{
\begin{tikzpicture}[scale=1]%unit=1cm
\foreach \X in {2} {
\foreach \Y in {2} {
\foreach \F in {1} {
\foreach \D in {0.5} {
\draw [thick, blue, domain=-\X:\X, samples=50]
plot (\x, {\F*sqrt((\D)^2 +(\x)^2)});
\draw [-] (\X*.5,\Y*1.1)--(\X*.5,\Y*1.1)
node[blue, pos=1, below] {\large $\alpha_{\rm f}(s)$};
%alpha(s)
\draw [->] (-\X,0)--(\X*1.1,0)
node[pos=1, right] {$s$};
\draw [->] (0,-\Y*0.1)--(0,\Y*1.1)
node[pos=0, left] {$0$} node[pos=0.4, left] {$\delta$};
\draw [densely dotted]
(0,0)--(-\X,\Y*\F)  (0,0)--(\X,\Y*\F);
%axes
}}}};
\end{tikzpicture}
}
\end{center}
\end{minipage}
\begin{minipage}[b]{.33\linewidth}
\begin{center}
\raisebox{0mm}{
\begin{tikzpicture}[scale=1]%unit=1cm
\foreach \X in {2} {
\foreach \Y in {2} {
\foreach \F in {1} {
\foreach \D in {0.5} {
\draw [thick, magenta, domain=-\X:\X, samples=50]
plot (\x, {\F*\x/sqrt((\D)^2 +(\x)^2)});
\draw [-] (\X*.6,\Y*.8)--(\X*.6,\Y*.8)
node[magenta, pos=1, below] {\large $\alpha^\prime_{\rm f}(s)$};
%alpha'(s)
\draw [->] (-\X,0)--(\X*1.1,0)
node[pos=1, right] {$s$};
\draw [->] (0,-\Y*.6)--(0,\Y*.6)
node[pos=0.55, left] {$0$};
\draw [densely dotted]
(-\X,-\Y*.5)--(\X,-\Y*.5)
node[pos=0.63, above] {$-K_{\rm f 1}$}
(-\X,\Y*.5)--(\X,\Y*.5)
node[pos=0.4, below] {$K_{\rm f 1}$};
%axes
}}}};
\end{tikzpicture}
}
\end{center}
\end{minipage}
\begin{minipage}[b]{.33\linewidth}
\begin{center}
\raisebox{0mm}{
\begin{tikzpicture}[scale=1]%unit=1cm
\foreach \X in {2} {
\foreach \Y in {2} {
\foreach \F in {1} {
\foreach \D in {0.5} {
\draw [thick, red, domain=-\X:\X, samples=50]
plot (\x, {\F*(\D)^2/(sqrt((\D)^2 +(\x)^2))^3});
\draw [-] (\X*.5,\Y)--(\X*.5,\Y)
node[red, pos=1, below] {\large $\alpha^{\prime\prime}_{\rm f}(s)$};
%alpha''(s)
\draw [->] (-\X,0)--(\X*1.1,0)
node[pos=1, right] {$s$};
\draw [->] (0,-\Y*.1)--(0,\Y*1.1)
node[pos=0, left] {$0$};
\draw [densely dotted] (-\X,\Y)--(\X,\Y)
node[pos=0.35, below] {$K_{\rm f 2}$};
%axes
}}}};
\end{tikzpicture}
}
\end{center}
\end{minipage}
\medskip
\caption{Example graphics of $\alpha_{\rm f}, \alpha^\prime_{\rm f},
\alpha^{\prime\prime}_{\rm f}$ in 1d.}\label{fig2}
\end{figure}
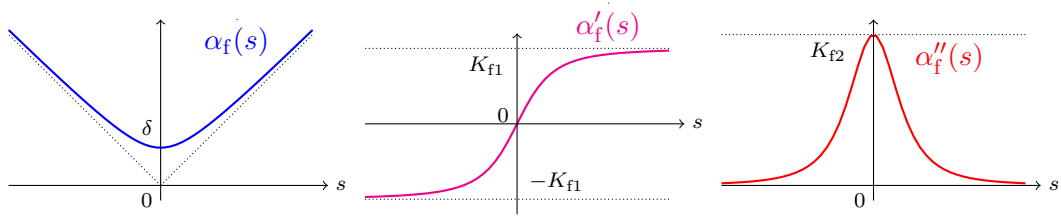

The latter term in \eqref{2.2} associates cohesion between the crack surfaces.
Let $s\mapsto\alpha_{\rm c}(s): \mathbb{R}\mapsto\mathbb{R}$
and its second derivative be uniformly continuous functions,
and let there exist constant $K_{\rm c}, K_{\rm c 1}, K_{\rm c 2}\ge0$ such that
\begin{equation}\label{2.5}
-K_{\rm c} |s|\le\alpha_{\rm c}(s),\quad |\alpha_{\rm c}^\prime(s)|\le K_{\rm c 1},
\quad |\alpha_{\rm c}^{\prime\prime}(s)|\le K_{\rm c 2}.
\end{equation}
From the physics literature (e.g. \cite{Kit/08}) we suggest
the following generic function
\begin{equation}\label{2.6}
\alpha_{\rm c}(s) =K_{\rm c}\, \frac{s}{\kappa +|s|^{m}},
\end{equation}
where $K_{\rm c}>0$ is related to the fracture toughness,
and $\kappa>0$, $m\ge1$ are parameters.
In this case, $K_{\rm c 1}$ and $K_{\rm c 2}$ are proportional to $K_{\rm c}$.
The example of $\alpha_{\rm c}, \alpha_{\rm c}^\prime,
\alpha_{\rm c}^{\prime\prime}$ for $m=4$ is depicted in Figure~\ref{fig3}.
\begin{figure}[hbt]
\begin{minipage}[b]{.33\linewidth}
\begin{center}
\raisebox{0mm}{
\begin{tikzpicture}[scale=1]%unit=1cm
\foreach \X in {2} {
\foreach \Y in {2} {
\foreach \G in {.65} {
%\foreach \G in {3.5} {
\foreach \K in {1} {
\foreach \M in {2} {
\draw [thick, blue, domain=-\X:\X, samples=50]
plot (\x, {\G*\x/(1 +\K*(\x)^(2*\M))});
\draw [-] (\X*.8,\Y*.5)--(\X*.8,\Y*.5)
node[blue, pos=1, below] {\large $\alpha_{\rm c}(s)$};
%alpha(s)
\draw [->] (-\X,0)--(\X*1.1,0)
node[pos=1, right] {$s$};
\draw [->] (0,-\Y)--(0,\Y*1.1)
node[pos=0.55, left] {$0$};
\draw [densely dotted]
%(-\X*.3,-\Y*\G*.3)--(\X*.3,\Y*\G*.3)
%(-\X,-\Y)--(\X,-\Y)
(-\X,-\Y*\G)--(\X,\Y*\G)
(-\X,-\Y*.2)--(\X,-\Y*.2)
%node[pos=0.65, above] {$-K_{1{\rm c}}$}
%(-\X,\Y)--(\X,\Y)
(-\X,\Y*.2)--(\X,\Y*.2)
%node[pos=0.4, below] {$K_{1{\rm c}}$}
;
%axes
}}}}};
\end{tikzpicture}
}
\end{center}
\end{minipage}
\begin{minipage}[b]{.33\linewidth}
\begin{center}
\raisebox{0mm}{
\begin{tikzpicture}[scale=1]%unit=1cm
\foreach \X in {2} {
\foreach \Y in {2} {
\foreach \G in {.65} {
%\foreach \G in {2} {
\foreach \K in {1} {
\foreach \M in {2} {
\draw [thick, magenta, domain=-\X:\X, samples=50]
plot (\x, {\G*(1-\K*(2*\M-1)*(\x)^(2*\M))/(1 +\K*(\x)^(2*\M))^2});
\draw [-] (\X*.6,\Y*.6)--(\X*.6,\Y*.6)
node[magenta, pos=1, below] {\large $\alpha_{\rm c}^\prime(s)$};
%alpha'(s)
\draw [->] (-\X,0)--(\X*1.1,0)
node[pos=1, right] {$s$};
\draw [->] (0,-\Y)--(0,\Y*1.1)
node[pos=0.55, left] {$0$};
\draw [densely dotted]
%(-\X,-\Y)--(\X,-\Y)
%node[pos=0.63, above] {$-K_{\rm c 1}$}
(-\X,-\Y*.33)--(\X,-\Y*.33)
node[pos=0.63, above] {$-K_{\rm c 1}$}
%(-\X,\Y)--(\X,\Y)
(-\X,\Y*.33)--(\X,\Y*.33)
node[pos=0.28, below] {$K_{\rm c 1}$};
%axes
}}}}};
\end{tikzpicture}
}
\end{center}
\end{minipage}
\begin{minipage}[b]{.33\linewidth}
\begin{center}
\raisebox{0mm}{
\begin{tikzpicture}[scale=1]%unit=1cm
\foreach \X in {2} {
\foreach \Y in {2} {
\foreach \G in {.65} {
\foreach \K in {1} {
\foreach \M in {2} {
\draw [thick, red, domain=-\X:\X, samples=50]
plot (\x,
{\G*2*\K*\M*(\x)^(2*\M-1)*(\K*(2*\M-1)*(\x)^(2*\M)-2*\M-1)
/(1 +\K*(\x)^(2*\M))^3});
\draw [-] (-\X*.7,\Y)--(-\X*.7,\Y)
node[red, pos=1, below] {\large $\alpha_{\rm c}^{\prime\prime}(s)$};
%alpha''(s)
\draw [->] (-\X,0)--(\X*1.1,0)
node[pos=1, right] {$s$};
\draw [->] (0,-\Y)--(0,\Y*1.1)
node[pos=0.55, left] {$0$};
\draw [densely dotted] (-\X,-\Y)--(\X,-\Y)
node[pos=0.35, above] {$-K_{\rm c 2}$}
(-\X,\Y)--(\X,\Y)
node[pos=0.6, below] {$K_{\rm c 2}$};
%axes
}}}}};
\end{tikzpicture}
}
\end{center}
\end{minipage}
\medskip
\caption{Example graphics of $\alpha_{\rm c}, \alpha_{\rm c}^\prime,
\alpha_{\rm c}^{\prime\prime}$ as $\kappa=1$ and $m=4$.}\label{fig3}
\end{figure}
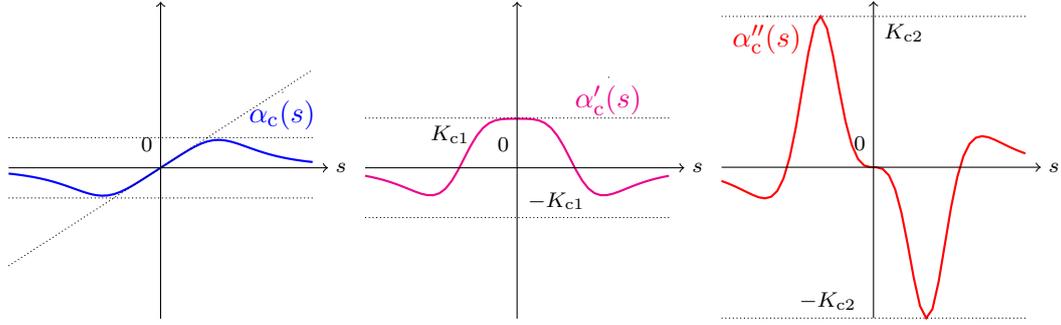
In particular, the left plot in Figure~\ref{fig3} depicts the typical softening phenomenon
for growing $s$.
It is worth noting that the left branch of $\alpha_{\rm c}(\nu_t\cdot[\![u]\!])$ for
$\nu_t\cdot[\![u]\!]<0$ implies a normal compliance and it is avoided
when the non-penetration $\nu_t\cdot[\![u]\!]\ge0$ in \eqref{2.1} holds.

The symmetric $d$-by-$d$ tensors of linearized strain $\epsilon$ and
the Cauchy stress $\sigma$ are given by
\begin{equation}\label{2.7}
\epsilon(u) =\frac{1}{2} (\nabla u +\nabla u^\top),
\quad \sigma(u) =C\epsilon(u),
\end{equation}
where $(\nabla u) =(\partial u_i/ \partial x_j)$ for $i,j =1,\ldots,d$,
the transposition $(\,\cdot\,)^\top$ swaps columns for rows.
A symmetric  fourth order tensor of elastic coefficients
$C(x)\in W^{1,\infty}(\Omega)^{d\times d\times d\times d}$, 
such that $C_{ijkl} =C_{jikl} =C_{klij}$ for $i, j, k, l =1,\ldots,d$, 
is positive definite and fulfills the Korn--Poincare inequality:
there exists $K_{\rm KP}>0$ such that
\begin{equation}\label{2.15}
\int_{\Omega\setminus\Sigma_t} \sigma(u)\cdot \epsilon(u) \,dx
\ge K_{\rm KP} \|u\|^2_{H^1(\Omega\setminus\Sigma_t)^d}
\quad\text{ for }u\in V(\Omega_t).
\end{equation}
over the Sobolev space
\begin{equation}\label{2.9}
V(\Omega_t) =\{u\in  H^1(\Omega^+_t)^d\cap H^1(\Omega^-_t)^d\vert
\quad u=0\text{ on }\Gamma^{\rm D}_t\}.
\end{equation}
For a boundary traction vector $g\in H^1(\partial\Omega)^d$,
we consider the following bulk energy
\begin{equation}\label{2.8}
\mathcal{B}(u; \Omega_t) =\frac{1}{2} \int_{\Omega\setminus\Sigma_t}
\sigma(u)\cdot \epsilon(u) \,dx -\int_{\Gamma^{\rm N}_t} g\cdot u \,dS_x.
\end{equation}
The feasible set corresponding to  the non-penetration condition in \eqref{2.1} reads
\begin{equation}\label{2.10}
K(\Omega_t) =\{u\in  V(\Omega_t)\vert
\quad \nu_t\cdot[\![u]\!]\ge0\text{ on }\Sigma_t\},
\end{equation}
which is a convex, closed cone.

\begin{theorem}[Well-posedness of cohesive crack problem]\label{theo1}
There exists a solution to the non-convex, constrained minimization problem:
find $u_t\in  K(\Omega_t)$ such that
\begin{equation}\label{2.11}
\mathcal{E}(u_t; \Omega_t) =\min_{u\in  K(\Omega_t)} \mathcal{E}(u; \Omega_t),
\end{equation}
where the total energy $\mathcal{E}$ according to \eqref{2.2} and \eqref{2.8} is given by
\begin{equation}\label{2.12}
\mathcal{E}(u; \Omega_t)
=\frac{1}{2} \int_{\Omega\setminus\Sigma_t} \sigma(u)\cdot \epsilon(u) \,dx
-\int_{\Gamma^{\rm N}_t} g\cdot u \,dS_x +\int_{\Sigma_t} \bigl\{
\alpha_{\rm f}([\![u]\!]_{\tau_t}) +\alpha_{\rm c}(\nu_t\cdot[\![u]\!]) \bigr\} \,dS_x.
\end{equation}
The solution satisfies the first-order optimality condition \eqref{1.2} in the form of VI:
\begin{multline}\label{2.13}
\int_{\Omega\setminus\Sigma_t} \sigma(u_t)\cdot \epsilon(u -u_t) \,dx
+\int_{\Sigma_t} \bigl\{
\nabla\alpha_{\rm f}([\![u_t]\!]_{\tau_t})\cdot [\![u -u_t]\!]_{\tau_t}\\
+\alpha_{\rm c}^\prime(\nu_t\cdot[\![u_t]\!])\, (\nu_t\cdot[\![u -u_t]\!])\bigr\} \,dS_x
\ge \int_{\Gamma^{\rm N}_t} g\cdot (u -u_t) \,dS_x
\end{multline}
for all test functions $u\in K(\Omega_t)$.
For smooth solutions the boundary value relations hold:
\begin{align}\label{2.14}
{\rm div}\, \sigma(u_t) =0 &\text{ in } \Omega\setminus\Sigma_t,\nonumber\\
u_t =0 \text{ on } \Gamma^{\rm D}_t,\quad
\sigma(u_t)n =g &\text{ on } \Gamma^{\rm N}_t,\nonumber\\
[\![\sigma(u_t)\nu_t]\!] =0,\quad (\sigma(u_t)\nu_t)_{\tau_t}
=\nabla\alpha_{\rm f}([\![u_t]\!]_{\tau_t}),&\nonumber\\
\nu_t\cdot[\![u_t]\!]\ge0, \quad \nu_t\cdot(\sigma(u_t)\nu_t)
\le\alpha_{\rm c}^\prime(\nu_t\cdot[\![u_t]\!]),&\nonumber\\
(\nu_t\cdot[\![u_t]\!])\, \bigl\{ \nu_t\cdot(\sigma(u_t)\nu_t)
-\alpha_{\rm c}^\prime(\nu_t\cdot[\![u_t]\!]) \bigr\} =0&\text{ on } \Sigma_t,
\end{align}
for the decomposition of vector $\sigma(u_t)\nu_t =\bigl( \nu_t\cdot (\sigma(u_t)\nu_t)
\bigr) \nu_t +(\sigma(u_t)\nu_t)_{\tau_t}$ according to \eqref{2.1}.
The last two lines in \eqref{2.14} are the complementarity conditions.
If both $\alpha_{\rm f}$ and $\alpha_{\rm c}$ were convex
(that is not $\alpha_{\rm c}$ in \eqref{2.6}),
then the solution $u_t$ to \eqref{2.11} and \eqref{2.13} would be unique.
\end{theorem}

\begin{proof}
On the right-hand side of \eqref{2.12}, the first, quadratic in $u$ 
integral term over $\Omega\setminus\Sigma_t$, is strongly positive 
by the Korn--Poincare inequality \eqref{2.15}. 
Using the Cauchy--Schwarz inequality, the other boundary integral terms over $\Sigma_t$
and $\Gamma^{\rm N}_t$ are bounded from below by a sub-linear in $u$ function
\begin{multline}\label{2.16}
\int_{\Sigma_t} \bigl\{
\alpha_{\rm f}([\![u]\!]_{\tau_t}) +\alpha_{\rm c}(\nu_t\cdot[\![u]\!]) \bigr\} \,dS_x
-\int_{\Gamma^{\rm N}_t} g\cdot u \,dS_x\\
\ge -\bigl(K_{\rm f} \|[\![u]\!]_{\tau_t}\|_{L^2(\Sigma_t)^d}
+K_{\rm c}  \|\nu_t\cdot[\![u]\!]\|_{L^2(\Sigma_t)} \bigr) \sqrt{|\Sigma_t|}
-\|g\|_{L^2(\Gamma^{\rm N}_t)^d} \|u\|_{L^2(\Gamma^{\rm N}_t)^d}
\end{multline}
by virtue of the properties for $\alpha_{\rm f}$, $\alpha_{\rm c}$ in \eqref{2.3}, \eqref{2.5}.
Therefore, estimating the jump by $\|[\![u]\!]\|^2_{L^2(\Sigma_t)}\le
2\|u\|^2_{L^2(\Sigma_t\cap \partial\Omega^+_t)}
+2\|u\|^2_{L^2(\Sigma_t\cap \partial\Omega^-_t)}$
and applying the trace inequality we have
\begin{equation}\label{2.17}
\|u\|_{L^2(\partial\Omega^\pm_t)^d}\le
\|u\|_{H^{1/2}(\partial\Omega^\pm_t)^d}\le K_{\rm tr} \|u\|_{H^1(\Omega^\pm_t)^d},
\quad u\in H^1(\Omega^\pm_t)^d.
\end{equation}
Then we get that $\mathcal{E}$ is radially unbounded, and thus coercive.
The functions $\alpha_{\rm f}$ and $\alpha_{\rm c}$ are uniformly continuous,
hence preserving $L^2$-convergence (see \cite{BJ/61}).
Using the compactness of the embedding of the traces of
$u$ at $\Sigma_t\cap\partial\Omega^\pm_t$,
from  $H^1(\Omega^\pm_t)$ into $L^2(\partial\Omega^\pm_t)$, it follows that
the mapping $u\mapsto\mathcal{E}(u)$ from $V(\Omega_t)\mapsto\mathbb{R}$
is weakly lower semi-continuous.

Let $\{u^n\}$, $n\in\mathbb{N}$, be an infimal sequence in $K(\Omega_t)$.
The coercivity of $\mathcal{E}$ implies the boundedness of $\{u^n\}$ in $V(\Omega_t)$.
Then, on a subsequence $\{u^{n_k}\}$, there exists an accumulation point
$u_t$ such that $u^{n_k}\rightharpoonup u_t$ weakly in
$H^1(\Omega\setminus\Sigma_t)^d$ as $n_k\to\infty$.
By weak closedness of $K(\Omega_t)$ we have $u_t\in K(\Omega_t)$.
Taking the limit inferior of $\mathcal{E}(u^{n_k})$, the weak lower semi-continuity
of $\mathcal{E}$ implies that $u_t$ attains the minimum in \eqref{2.11}.
Applying standard variational arguments implies the optimality condition \eqref{2.13}
and \eqref{2.14}, see details in \cite[Section~1.4]{KK/00}.
Moreover, if $\alpha_{\rm f}$, $\alpha_{\rm c}$ were convex,
then the integral over $\Sigma_t$ in \eqref{2.12} is monotone.
This would lead to uniqueness of $u_t$ as solution to \eqref{2.13},
which is then necessarily the unique solution for \eqref{2.11}.
\qed
\end{proof}

Next we approximate the VI \eqref{2.13} by a penalty method.
By itself penalization is a self-contained physical model allowing compliance,
see \cite{And/99} for the discussion.

\section{Lavrentiev based regularization and saddle-point problem}\label{sec3}

Let $\varepsilon_0 >0$.
For $\varepsilon\in (0, \varepsilon_0)$,
the standard penalization of the inequality constraint $s\ge0$ by
$-[s]^-/\varepsilon$ has only a generalized derivative $\mathcal{H}(-s)/\varepsilon$,
where $\mathcal{H}$ is the Heaviside step function such that
$\mathcal{H}(s)=1$ for $s>0$, otherwise $\mathcal{H}(s)=0$ for $s\le 0$.
We suggest a Lavrentiev based $C^1$-regularization
by the normal compliance $\beta_\varepsilon$ as follows.
Let the function $s\mapsto\beta_\varepsilon(s): \mathbb{R}\mapsto\mathbb{R}$
be concave and differentiable, with $\beta$ and $\beta^\prime$ uniformly continuous,
and let there exist $K_{\beta}, K_{\beta 1}\ge0$ such that
\begin{equation}\label{3.1}
\bigl| \beta_\epsilon(s) +\frac{[s]^-}{\varepsilon} \bigr|
\le K_{\beta},\quad
0\le \beta^\prime_\epsilon(s)\le \frac{K_{\beta 1}}{\varepsilon}.
\end{equation}
We assume that the following conditions hold, which describe relaxed complementarity
and compliance, respectively:
\begin{equation}\label{3.2}
\beta_\epsilon(s) [s]^+\ge -\varepsilon K_{\beta},\quad
\beta_\epsilon(s) [s]^-\le -\frac{([s]^-)^2}{\varepsilon}
+\varepsilon K_{\beta}.
\end{equation}
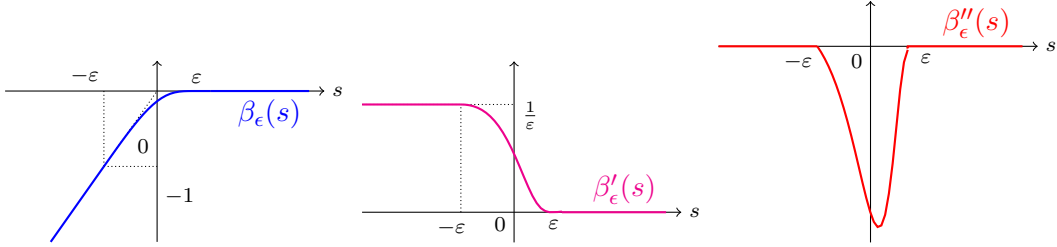
\begin{figure}[hbt]
\begin{minipage}[b]{.33\linewidth}
\begin{center}
\raisebox{0mm}{
\begin{tikzpicture}[scale=1]%unit=1cm
\foreach \X in {2} {
\foreach \Y in {2} {
\foreach \E in {.7} {
\draw [->] (-\X,0)--(\X*1.1,0)
node[pos=1, right] {$s$}
node[pos=.25, above] {$-\varepsilon$}
node[pos=.6, above] {$\varepsilon$};
\draw [->] (0,-\Y)--(0,\Y*.2)
node[pos=.53, left] {$0$}
node[pos=.25, right] {$-1$};
\draw [densely dotted]
(-\E,-1)--(0,0) (-\E,0)--(-\E,-1) (-\E,-1)--(0,-1);
%axes
\draw [thick, blue,-] (-\E*\Y,-\Y)--(-\E,-1);
\draw [thick, blue, domain=-\E:\E, samples=25]
plot (\x, {-exp(2*(\x+\E)/(\x-\E))});
\draw [thick, blue,-] (\E,0)--(\X,0)
node[blue, pos=.6, below] {\large $\beta_\epsilon(s)$};
%beta'(s)
}}};
\end{tikzpicture}
}
\end{center}
\end{minipage}
\begin{minipage}[b]{.33\linewidth}
\begin{center}
\raisebox{0mm}{
\begin{tikzpicture}[scale=1]%unit=1cm
\foreach \X in {2} {
\foreach \Y in {2} {
\foreach \E in {.7} {
\draw [->] (-\X,0)--(\X*1.1,0)
node[pos=1, right] {$s$}
node[pos=.28, below] {$-\varepsilon$}
node[pos=.6, below] {$\varepsilon$};
\draw [->] (0,-\Y*.2)--(0,\Y*1)
node[pos=.1, left] {$0$}
node[pos=.7, right] {$\frac{1}{\varepsilon}$};
\draw [densely dotted]
(-\X,1/\E)--(0,1/\E) (-\E,0)--(-\E,1/\E);
%axes
\draw [thick, magenta,-] (-\X,1/\E)--(-\E,1/\E);
\draw [thick, magenta, domain=-\E:\E*.9, samples=25]
plot (\x, {4*\E/(\x-\E)^2*exp(2*(\x+\E)/(\x-\E))});
\draw [thick, magenta,-] (\E*.9,0)--(\X,0)
node[magenta, pos=.6, above] {\large $\beta^\prime_\epsilon(s)$};
%beta''(s)
}}};
\end{tikzpicture}
}
\end{center}
\end{minipage}
\begin{minipage}[b]{.33\linewidth}
\begin{center}
\raisebox{0mm}{
\begin{tikzpicture}[scale=1]%unit=1cm
\foreach \X in {2} {
\foreach \Y in {2} {
\foreach \E in {.7} {
\draw [->] (-\X,0)--(\X*1.1,0)
node[pos=1, right] {$s$}
node[pos=.25, below] {$-\varepsilon$}
node[pos=.65, below] {$\varepsilon$};
\draw [->] (0,-\Y*1.3)--(0,\Y*0.3)
node[pos=.75, left] {$0$};
%axes
\draw [thick, red, -] (-\X,0)--(-\E,0);
\draw [thick, red, domain=-\E:\E*.7, samples=25]
plot (\x, {-8*\E*(\x+\E)*exp(2*(\x+\E)/(\x-\E))/((\x-\E)^2)/((\x-\E)^2)});
\draw [thick, red,-] (\E*.69,-0.03)--(\E*.7,0)--(\X,0)
node[red, pos=.6, above] {\large $\beta^{\prime\prime}_\epsilon(s)$};
%beta(s)
}}};
\end{tikzpicture}
}
\end{center}
\end{minipage}
\medskip
\caption{Example graphics of $\beta_\varepsilon, \beta^\prime_\varepsilon,
\beta^{\prime\prime}_\varepsilon$ for fixed $\varepsilon$.}\label{fig4}
\end{figure}
For example, we construct the following mollification of minimum function
\begin{equation}\label{3.3}
\beta_\epsilon(s) =\begin{cases}
s/\varepsilon & \text{ for } s<-\varepsilon\\
-\exp\bigl( 2(s +\varepsilon)/(s -\varepsilon) \bigr)
& \text{ for } -\varepsilon\le s<\varepsilon\\
0 & \text{ for } s\ge\varepsilon
\end{cases}
\end{equation}
which is depicted in Figure~\ref{fig4} together with its two derivatives.

\begin{lemma}\label{lem0}
For $\beta_\varepsilon$ from \eqref{3.3}, the properties \eqref{3.1} and \eqref{3.2}
hold true with $K_{\beta} =K_{\beta 1} =1$.
Moreover,  $\beta^{\prime\prime}_\varepsilon\le0$ implies that
$\beta^\prime_\varepsilon\ge0$ decreases monotonically,
and $\beta_\varepsilon\le0$ is concave and increases monotonically.
\end{lemma}

\begin{proof}
The properties \eqref{3.1} can be easily checked.
To verify the first inequality in \eqref{3.2}, from \eqref{3.3}
we deduce that $\beta_\epsilon(s) [s]^+ =0$ for $s\ge \varepsilon$.
Here we use the complementary condition $[s]^- [s]^+ =0$
and $[s]^+ =0$ for $s<0$.
We further have $\beta_\epsilon(s)\ge -[s]^-/\varepsilon -K_{\beta}$
according to the first estimate in \eqref{3.1}.
Henceforth, after multiplication with $[s]^+\in [0,\varepsilon)$,
the lower bound $\beta_\epsilon(s) [s]^+ \ge -\varepsilon K_{\beta}$
holds for $0\le s <\varepsilon$.

Similarly, $\beta_\epsilon(s) [s]^- =-([s]^-)^2/\varepsilon$
for $s< -\varepsilon$ in \eqref{3.3},
and $\beta_\epsilon(s) [s]^- =0$ due to $[s]^- =0$ for $s\ge 0$.
The first estimate in \eqref{3.1}, that is
$\beta_\epsilon(s) \le -[s]^-/\varepsilon +K_{\beta}$,
after multiplication with $[s]^-\in(0,\varepsilon]$
leads to the upper bound $\beta_\epsilon(s) [s]^-\le
-([s]^-)^2/\varepsilon +\varepsilon K_{\beta}$ for $-\varepsilon\le s <0$.
This proves the second inequality in \eqref{3.2}.
\qed
\end{proof}

Using Lemma~\ref{lem0} we obtain the existence result for the penalized cohesive crack problem.

\begin{theorem}[Well-posedness of $\varepsilon$-regularized cohesive crack problem]\label{theo2}
There exists a solution to the penalty problem:
find $u^\varepsilon_t\in  V(\Omega_t)$ such that
\begin{multline}\label{3.4}
\int_{\Omega\setminus\Sigma_t} \sigma(u^\varepsilon_t)\cdot \epsilon(u) \,dx
+\int_{\Sigma_t} \bigl\{
\nabla\alpha_{\rm f}([\![u^\varepsilon_t]\!]_{\tau_t})\cdot [\![u]\!]_{\tau_t}\\
+[\alpha_{\rm c}^\prime +\beta_\varepsilon] (\nu_t\cdot[\![u^\varepsilon_t]\!])\,
(\nu_t\cdot[\![u]\!])\bigr\} \,dS_x =\int_{\Gamma^{\rm N}_t} g\cdot u \,dS_x
\end{multline}
for all test functions $u\in V(\Omega_t)$.
For smooth solutions the boundary value relations hold:
\begin{align}\label{3.5}
{\rm div}\, \sigma(u^\varepsilon_t) =0 &\text{ in } \Omega\setminus\Sigma_t,\nonumber\\
u^\varepsilon_t =0 \text{ on } \Gamma^{\rm D}_t,\quad
\sigma(u^\varepsilon_t)n =g &\text{ on } \Gamma^{\rm N}_t,\nonumber\\
[\![\sigma(u^\varepsilon_t)\nu_t]\!] =0,\; (\sigma(u^\varepsilon_t)\nu_t)_{\tau_t}
=\nabla\alpha_{\rm f}([\![u^\varepsilon_t]\!]_{\tau_t}),\;
\nu_t\cdot(\sigma(u^\varepsilon_t)\nu_t) =[\alpha_{\rm c}^\prime
+\beta_\varepsilon] (\nu_t\cdot[\![u^\varepsilon_t]\!])&\text{ on } \Sigma_t.
\end{align}
If both $\nabla\alpha_{\rm f}$ and $\alpha_{\rm c}^\prime$ were monotone,
then the solution $u^\varepsilon_t$ to \eqref{3.4} would be unique.
\end{theorem}

\begin{proof}
We apply arguments similar to those in the proof of Theorem~\ref{theo1}.
From the properties of $\nabla\alpha_{\rm f}$ in \eqref{2.3}, and
$\alpha_{\rm c}^\prime$ from \eqref{2.5}, the fact that
$\beta_\epsilon(s) s\ge ([s]^-)^2/\varepsilon -2\varepsilon K_{\beta}$ by \eqref{3.2},
and using the Cauchy--Schwarz, Korn--Poincare \eqref{2.15}, and trace inequalities
\eqref{2.17}, similarly to \eqref{2.16} we deduce the uniform lower bound
\begin{multline}\label{3.6}
\int_{\Omega\setminus\Sigma_t} \sigma(u)\cdot \epsilon(u) \,dx
+\int_{\Sigma_t} \bigl\{
\nabla\alpha_{\rm f}([\![u]\!]_{\tau_t})\cdot [\![u]\!]_{\tau_t}
+[\alpha_{\rm c}^\prime +\beta_\varepsilon] (\nu_t\cdot[\![u]\!])\,
(\nu_t\cdot[\![u]\!])\bigr\} \,dS_x\\
-\int_{\Gamma^{\rm N}_t} g\cdot u \,dS_x
\ge K_{\rm KP} \|u\|^2_{H^1(\Omega\setminus\Sigma_t)^d}
-K_{t {\rm f c} 1} \|u\|_{H^1(\Omega\setminus\Sigma_t)^d}
-2\varepsilon K_{\beta} |\Sigma_t|,
\end{multline}
where
\begin{equation}\label{3.7}
K_{t {\rm f c} 1} :=\bigl( \|g\|_{L^2(\Gamma^{\rm N}_t)^d}
+(K_{\rm f 1} +K_{\rm c 1}) \sqrt{2|\Sigma_t|} \bigr) K_{\rm tr}.
\end{equation}
Therefore, the operator associated to \eqref{3.4}, denoted following \eqref{1.5}
by $\partial^\varepsilon_u \mathcal{E}: V(\Omega_t)\mapsto V(\Omega_t)^\star$,
is coercive.
We have $\nabla\alpha_{\rm f}$ and $[\alpha_{\rm c}^\prime +\beta_\varepsilon]$
are uniformly continuous, and thus preserve $L^2$-convergence,
the operator $\partial^\varepsilon_u \mathcal{E}$ is weakly continuous in the following sense.
If $u^n\rightharpoonup u_t$ weakly in $H^1(\Omega\setminus\Sigma_t)^d$
as $n\to\infty$ (hence $u^n\to u_t$ strongly in
$L^2(\partial\Omega\cup\Sigma^\pm_t)^d$ by compactness),
then for each $u\in V(\Omega_t)$ the following convergence holds
\begin{multline*}
\int_{\Omega\setminus\Sigma_t} \sigma(u^n)\cdot \epsilon(u) \,dx
+\int_{\Sigma_t} \bigl\{
\nabla\alpha_{\rm f}([\![u^n]\!]_{\tau_t})\cdot [\![u]\!]_{\tau_t}
+[\alpha_{\rm c}^\prime +\beta_\varepsilon] (\nu_t\cdot[\![u^n]\!])\,
(\nu_t\cdot[\![u]\!])\bigr\} \,dS_x\\
\to \int_{\Omega\setminus\Sigma_t} \sigma(u_t)\cdot \epsilon(u) \,dx
+\int_{\Sigma_t} \bigl\{
\nabla\alpha_{\rm f}([\![u_t]\!]_{\tau_t})\cdot [\![u]\!]_{\tau_t}
+[\alpha_{\rm c}^\prime +\beta_\varepsilon] (\nu_t\cdot[\![u_t]\!])\,
(\nu_t\cdot[\![u]\!])\bigr\} \,dS_x.
\end{multline*}
Therefore, applying a Galerkin approximation and the Brouwer fixed point theorem
(see \cite{Fra/94}), a solution to the variational problem \eqref{3.4} can be argued.
Its uniqueness under the monotony assumption 
(that is not $\alpha_{\rm c}^\prime$ in \eqref{2.6}), 
and the boundary value formulation \eqref{3.5} can be derived in a standard way.
\qed
\end{proof}

Next, for a given observation $z\in H^1(\partial\Omega)^d$, we consider
the $\varepsilon$-dependent least-squares misfit function from \eqref{1.4},
where $u^\varepsilon_t$ satisfies \eqref{3.4}:
\begin{equation}\label{3.8}
\mathcal{J}(u^\varepsilon_t; \Omega_t)
=\frac{1}{2} \int_{\Gamma^{\rm O}_t} |u^\varepsilon_t - z|^2 \,dS_x
+\rho |\Sigma_t|.
\end{equation}
From the fundamental theorem of calculus, we have the following representations
\begin{align}\label{3.9}
\nabla\alpha_{\rm f}([\![u]\!]_{\tau_t})
=\int_0^1 \nabla^2 \alpha_{\rm f} ([\![r u]\!]_{\tau_t})  \,
[\![u]\!]_{\tau_t} \,dr +\nabla\alpha_{\rm f} (0), \nonumber\\
[\alpha_{\rm c}^\prime +\beta_\varepsilon](\nu_t\cdot[\![u]\!])
=\int_0^1 [\alpha_{\rm c}^{\prime\prime} +\beta^\prime_\varepsilon]
(\nu_t\cdot[\![r u]\!]) \,(\nu_t\cdot[\![u]\!]) \,dr
+[\alpha_{\rm c}^\prime +\beta_\varepsilon](0)
\end{align}
for differentiable $\nabla\alpha_{\rm f}, \alpha_{\rm c}^\prime, \beta_\varepsilon$.
Let us fix a solution $u^\varepsilon_t$ to the variational equation \eqref{3.4}.
Based on \eqref{3.9} we introduce a quadratic Lagrangian (compare to
$\mathcal{L}^\varepsilon$ in \eqref{1.6}) linearized around $u^\varepsilon_t$
\begin{multline}\label{3.10}
\tilde{\mathcal{L}}^\varepsilon (0, u^\varepsilon_t, u, v; \Omega_t)
=\frac{1}{2} \int_{\Gamma^{\rm O}_t} |u - z|^2 \,dS_x +\rho |\Sigma_t|
-\int_{\Omega\setminus\Sigma_t} \sigma(u)\cdot \epsilon(v) \,dx
+\int_{\Gamma^{\rm N}_t} g\cdot v \,dS_x\\ -\int_{\Sigma_t} \Bigl\{
\Bigl( \int_0^1 \nabla^2 \alpha_{\rm f} ([\![r u^\varepsilon_t]\!]_{\tau_t})  \,
[\![u]\!]_{\tau_t} \,dr +\nabla\alpha_{\rm f}(0) \Bigr) \cdot [\![v]\!]_{\tau_t}\\
+\Bigl( \int_0^1 [\alpha_{\rm c}^{\prime\prime} +\beta^\prime_\varepsilon]
(\nu_t\cdot[\![r u^\varepsilon_t]\!]) \,(\nu_t\cdot[\![u]\!]) \,dr
+[\alpha_{\rm c}^\prime +\beta_\varepsilon](0) \Bigr) (\nu_t\cdot[\![v]\!]) \Bigr\} \,dS_x,
\end{multline}
and a saddle point problem corresponding  to \eqref{1.11}: for all $(u, v)\in V(\Omega_t)^2$,
\begin{equation}\label{3.11}
\tilde{\mathcal{L}}^\varepsilon (0, u^\varepsilon_t, u^\varepsilon_t, v;\Omega_t) \le
\tilde{\mathcal{L}}^\varepsilon (0, u^\varepsilon_t, u^\varepsilon_t, v^\varepsilon_t;\Omega_t)
\le\tilde{\mathcal{L}}^\varepsilon (0, u^\varepsilon_t, u, v^\varepsilon_t;\Omega_t).
\end{equation}
Then \eqref{3.8} can be expressed equivalently in the primal-dual form
\eqref{1.11} as
\begin{equation}\label{3.12}
\mathcal{J}(u^\varepsilon_t; \Omega_t) =\tilde{\mathcal{L}}^\varepsilon
(0, u^\varepsilon_t, u^\varepsilon_t, v^\varepsilon_t;\Omega_t),
\end{equation}
where according to \eqref{3.9} the optimal value of the Lagrangian at the solution is
\begin{multline}\label{3.13}
\tilde{\mathcal{L}}^\varepsilon (0, u^\varepsilon_t, u^\varepsilon_t, v^\varepsilon_t; \Omega_t)
=\frac{1}{2} \int_{\Gamma^{\rm O}_t} |u^\varepsilon_t - z|^2 \,dS_x +\rho |\Sigma_t|
-\int_{\Omega\setminus\Sigma_t} \sigma(u^\varepsilon_t)\cdot \epsilon(v^\varepsilon_t) \,dx
+\int_{\Gamma^{\rm N}_t} g\cdot v^\varepsilon_t \,dS_x\\ -\int_{\Sigma_t} \bigl\{
\nabla\alpha_{\rm f}([\![u^\varepsilon_t]\!]_{\tau_t}) \cdot [\![v^\varepsilon_t]\!]_{\tau_t}
+[\alpha_{\rm c}^\prime +\beta_\varepsilon](\nu_t\cdot[\![u^\varepsilon_t]\!])\,
(\nu_t\cdot[\![v^\varepsilon_t]\!]) \bigr\} \,dS_x.
\end{multline}

\begin{theorem}[Well-posedness of $\varepsilon$-regularized saddle-point problem]\label{theo3}
Assume that the cohesion is small in the sense that constant
$K_{\rm f 2}, K_{\rm c 2}$ in \eqref{2.3}, \eqref{2.5} are sufficiently small so that
\begin{equation}\label{3.14}
K_{\rm f c 2} := K_{\rm KP} -(K_{\rm f 2} +K_{\rm c 2}) 2 K_{\rm tr}^2 >0,
\end{equation}
where $K_{\rm KP}$, $K_{\rm tr}$ are from \eqref{2.15}, \eqref{2.17}.
Then there exists a unique saddle-point $(u^\varepsilon_t, v^\varepsilon_t)\in V(\Omega_t)^2$
in \eqref{3.11}.
Its primal component $u^\varepsilon_t$ solves \eqref{3.4}.
The dual component $v^\varepsilon_t$ is a solution to the adjoint equation
corresponding to fixed $u^\varepsilon_t$:
\begin{multline}\label{3.15}
\langle A_{\varepsilon}(u^\varepsilon_t) v, v^\varepsilon_t \rangle :=
\int_{\Omega\setminus\Sigma_t} \sigma(v)\cdot \epsilon(v^\varepsilon_t) \,dx
+\int_{\Sigma_t} \int_0^1 \Bigl\{ \Bigl( \nabla^2 \alpha_{\rm f}
([\![r u^\varepsilon_t]\!]_{\tau_t})  \, [\![v]\!]_{\tau_t} \Bigr)
\cdot [\![v^\varepsilon_t]\!]_{\tau_t}\\
+[\alpha_{\rm c}^{\prime\prime} +\beta^\prime_\varepsilon]
(\nu_t\cdot[\![r u^\varepsilon_t]\!]) \,(\nu_t\cdot[\![v]\!])
(\nu_t\cdot[\![v^\varepsilon_t]\!]) \Bigr\} \,dr \,dS_x
=\int_{\Gamma^{\rm O}_t} (u^\varepsilon_t -z)\cdot v \,dS_x
\end{multline}
for all test functions $v\in V(\Omega_t)$.
For smooth solutions the boundary value relations hold:
\begin{align}\label{3.16}
{\rm div}\, \sigma(v^\varepsilon_t) =0 &\text{ in } \Omega\setminus\Sigma_t,\nonumber\\
v^\varepsilon_t =0 \text{ on } \Gamma^{\rm D}_t,\quad
\sigma(v^\varepsilon_t)n =u^\varepsilon_t -z \text{ on } \Gamma^{\rm O}_t,
\quad \sigma(v^\varepsilon_t)n =0 &\text{ on }
\Gamma^{\rm N}_t\setminus \Gamma^{\rm O}_t,\nonumber\\
[\![\sigma(v^\varepsilon_t)\nu_t]\!] =0,\quad (\sigma(v^\varepsilon_t)\nu_t)_{\tau_t}
=\int_0^1 \nabla^2 \alpha_{\rm f} ([\![r u^\varepsilon_t]\!]_{\tau_t})
\, [\![v^\varepsilon_t]\!]_{\tau_t} \,dr,&\nonumber\\
\nu_t\cdot(\sigma(v^\varepsilon_t)\nu_t)
=\int_0^1 [\alpha_{\rm c}^{\prime\prime} +\beta^\prime_\varepsilon]
(\nu_t\cdot[\![r u^\varepsilon_t]\!]) \,(\nu_t\cdot[\![v^\varepsilon_t]\!]) \,dr
&\text{ on } \Sigma_t
\end{align}
implying linear, Robin-type boundary conditions at the interface.
\end{theorem}

\begin{proof}
The saddle-point problem consists of two sub-problems:
the former and the latter inequalities in \eqref{3.11}.
Since the Lagrangian $\tilde{\mathcal{L}}^\varepsilon$ from \eqref{3.10} is linear in $v$,
the primal maximization problem (the former inequality in \eqref{3.11})
is equivalent to the first order optimality condition \eqref{3.4}.
Its solvability is proven in Theorem~\ref{theo2}.
Since $\tilde{\mathcal{L}}^\varepsilon$ from \eqref{3.10}
is quadratic and convex in $u$,  the dual minimization problem
(the latter inequality in \eqref{3.11}) is the optimality condition
expressed by the adjoint equation \eqref{3.15}.

Now we prove the solution existence for \eqref{3.15}.
For fixed $u^\varepsilon_t$, the left-hand side of \eqref{3.15} forms
a linear  continuous operator
$A_{\varepsilon}(u^\varepsilon_t): V(\Omega_t)\mapsto V^\star(\Omega_t)$.
Indeed, using the Cauchy--Schwarz inequality and the upper bounds for
$\nabla^2 \alpha_{\rm f}$, $\alpha_{\rm c}^{\prime\prime}$, $\beta^\prime_\varepsilon$
in \eqref{2.3}, \eqref{2.5}, \eqref{3.1}, the operator is bounded from above,
hence continuous.
Recalling the symmetry of the elasticity coefficients $C$ and
the Hessian matrix $\nabla^2 \alpha_{\rm f}$, the operator is self-adjoint.
Applying the Cauchy--Schwarz, Korn--Poincare \eqref{2.15} and trace inequalities
\eqref{2.17}, due to the boundedness of $\nabla^2 \alpha_{\rm f}$,
$\alpha_{\rm c}^{\prime\prime}$, $\beta^\prime_\varepsilon\ge0$ in \eqref{2.3},
\eqref{2.5}, \eqref{3.1}, similarly to \eqref{3.6}, we estimate uniformly from below
\begin{multline}\label{3.17}
\langle A_{\varepsilon}(u^\varepsilon_t) u, u \rangle
\ge K_{\rm KP} \|u\|^2_{H^1(\Omega\setminus\Sigma_t)^d}
-\int_{\Sigma_t} \bigl\{ K_{\rm f 2} \bigl| [\![u]\!]_{\tau_t} \bigr|^2
+K_{\rm c 2} \bigl| \nu_t\cdot[\![u]\!] \bigr|^2  \bigr\} \,dS_x\\
\ge K_{\rm f c 2} \|u\|^2_{H^1(\Omega\setminus\Sigma_t)^d}.
\end{multline}
Here $K_{\rm f c 2} >0$ due to assumption \eqref{3.14}.
In this case, $A_{\varepsilon}(u^\varepsilon_t)$ is uniformly positive.
Because $\nabla^2\alpha_{\rm f}$ and $[\alpha_{\rm c}^{\prime\prime}
+\beta_\varepsilon^\prime]$ are assumed uniformly continuous,
they preserve $L^2$-convergence, and the operator 
$A_{\varepsilon}(u^\varepsilon_t)$ is weakly lower semi-continuous by 
the compactness similar to arguments presented in the proof of Theorem~\ref{theo2}.
According to the Lax--Milgram theorem, the variational equation \eqref{3.15}
has a unique solution.
We derive straightforwardly its boundary value formulation \eqref{3.16}.

Since the variational equation \eqref{3.4} can be rewritten in the equivalent form
\begin{equation*}
\langle A_{\varepsilon}(u^\varepsilon_t) u^\varepsilon_t, u \rangle
+\int_{\Sigma_t} \bigl( \nabla \alpha_{\rm f}(0)\cdot [\![u]\!]_{\tau_t}
+[\alpha_{\rm c}^\prime +\beta_\varepsilon](0) (\nu_t\cdot[\![u]\!]) \bigr) \,dS_x
=\int_{\Gamma^{\rm N}_t} g\cdot u \,dS_x
\end{equation*}
for all $u\in V(\Omega_t)$,
by assumption \eqref{3.14} its solution $u^\varepsilon_t$ is unique, too.
\qed
\end{proof}

\section{Shape derivative}\label{sec4}

Let us fix a  flow and its inverse
\begin{equation}\label{4.1}
s\mapsto (\phi_s, \phi^{-1}_s)\in C^{1}([t_0-t_1, t_1-t_0];
W^{1,\infty}(\overline{\Omega})^d)^2.
\end{equation}
This defines an associated coordinate transformation  $y= \phi_s(x)$ and its inverse  $x= \phi^{-1}_s(y)$.
For every fixed $t\in(t_0,t_1)$, we suppose that for $s\in[t_0, t_1] -t$ it forms a diffeomorphism
\begin{equation}\label{4.2}
\phi_s:\Omega_{t}\mapsto \Omega_{t+s},\; x\mapsto y,\quad
\phi^{-1}_s:\Omega_{t+s}\mapsto \Omega_{t},\; y\mapsto x,
\end{equation}
where the perturbed geometry $\Omega_{t+s} =(\Gamma^{\rm D}_{t+s},
\Gamma^{\rm N}_{t+s}, \Gamma^{\rm O}_{t+s}, \Sigma_{t+s})$ 
describes the broken domain $\Omega\setminus\Sigma_{t+s}$. 
From \eqref{4.1}, a time-dependent kinematic velocity $\Lambda(t, x)\in
C([t_0, t_1];W^{1,\infty}(\overline{\Omega})^d)$ is assumed
defined by the formula
\begin{equation}\label{4.3}
\Lambda(t+s,y) :={\textstyle\frac{d}{ds}}\phi_s(\phi^{-1}_s(y)).
\end{equation}
If a stationary velocity is given explicitly by
$\Lambda(x)\in W^{1,\infty}(\overline{\Omega})^d$
with $n\cdot\Lambda =0$ at $\partial\Omega$, thus preserving the hold-all domain,
then $\Lambda$ determines the flow \eqref{4.1} by unique solutions
to the autonomous ODE systems:
\begin{equation}\label{4.4}
\left\{ \begin{array}{rl}
{\textstyle\frac{d}{ds}}\phi_s =\Lambda(\phi_s)
&\;\text{for $s\not=0$},\\
\phi_s=x&\;\text{for $s=0$},
\end{array}\right.\quad
\left\{ \begin{array}{rl}
{\textstyle\frac{d}{ds}}\phi^{-1}_s =-\Lambda(\phi^{-1}_s)
&\;\text{for $s\not=0$},\\
\phi^{-1}_s =y&\;\text{for $s=0$},
\end{array}\right.
\end{equation}
which build a semi-group of transformations.

The following properties (T1)--(T4) are needed to prove shape differentiability.

\begin{description}
\item[(T1)]
We assume that
the map $u\mapsto u\circ\phi_s$ is bijective between the function spaces
\begin{equation}\label{4.5}
V(\Omega_{t+s})\mapsto V(\Omega_{t}).
\end{equation}
\end{description}

Based on assumption \eqref{4.5}, the perturbed objective
$(t_0-t, t_1-t)\times V(\Omega_{t})$,
$(s, \tilde{u})\mapsto \tilde{\mathcal{J}}$
and Lagrangian  $(t_0-t, t_1-t)\times V(\Omega_{t})^2$,
$(s, \tilde{u}, \tilde{v})\mapsto \tilde{\mathcal{L}}^\varepsilon$,
are well-defined  for $(u, v)\in V(\Omega_{t+s})^2$ when
transformed to the reference geometry $\Omega_{t}$ by setting
\begin{equation}\label{4.6}
\tilde{\mathcal{J}}(s, u\circ \phi_s; \Omega_{t}) =\mathcal{J}(u; \Omega_{t+s}),
\quad \tilde{\mathcal{L}}^\varepsilon(s, u\circ \phi_s, u\circ \phi_s, v\circ \phi_s; \Omega_{t})
=\mathcal{L}^\varepsilon(u, v; \Omega_{t+s}).
\end{equation}
At $s=0$ relations \eqref{4.6} imply that for $(\tilde{u}, \tilde{v})\in V(\Omega_{t})^2$
\begin{equation}\label{4.7}
\tilde{\mathcal{J}} (0, \tilde{u}; \Omega_{t}) =\mathcal{J}(\tilde{u}; \Omega_{t}),
\quad \tilde{\mathcal{L}}^\varepsilon (0, \tilde{u}, \tilde{u}, \tilde{v}; \Omega_{t})
=\mathcal{L}^\varepsilon (\tilde{u}, \tilde{v}; \Omega_{t}).
\end{equation}
According to \eqref{1.11} we look for a saddle-point
$(\tilde{u}^\varepsilon_{t+s}, \tilde{v}^\varepsilon_{t+s}) \in V(\Omega_{t})^2$
satisfying the inequalities
\begin{equation}\label{4.8}
\tilde{\mathcal{L}}^\varepsilon(s, u^\varepsilon_t, \tilde{u}^\varepsilon_{t+s}, \tilde{v};
\Omega_{t}) \le \tilde{\mathcal{L}}^\varepsilon(s, u^\varepsilon_t, \tilde{u}^\varepsilon_{t+s},
\tilde{v}^\varepsilon_{t+s}; \Omega_{t}) \le \tilde{\mathcal{L}}^\varepsilon
(s, u^\varepsilon_t, \tilde{u}, \tilde{v}^\varepsilon_{t+s}; \Omega_{t})
\end{equation}
for all $(\tilde{u}, \tilde{v})\in V(\Omega_{t})^2$.
In the case of $\mathcal{J}$ from \eqref{3.8}, applying the coordinate transformation
\eqref{4.2} we derive explicitly the objective function
\begin{equation}\label{4.9}
\tilde{\mathcal{J}}(s, \tilde{u}; \Omega_{t}) =\frac{1}{2} \int_{\Gamma^{\rm O}_{t}}
|\tilde{u} -z\circ \phi_s|^2 \,\omega^{\rm b}_s dS_x +\rho \int_{\Sigma_{t}} \omega^{\rm b}_s \,dS_x,
\end{equation}
where 
\begin{equation}\label{4.13}
\omega^{\rm d}_s:=\det(\nabla \phi_s)\text{ in }\Omega\setminus\Sigma_{t},
\quad \omega^{\rm b}_s :=|(\nabla\phi_s^{-\top}\circ \phi_s)
n^\pm_{t}| \omega^{\rm d}_s\text{ at }\partial\Omega^\pm_t
\end{equation}
denote the Jacobians, and set the perturbed Lagrangian according to \eqref{3.10} as
\begin{multline}\label{4.10}
\tilde{\mathcal{L}}^\varepsilon(s, u^\varepsilon_t, \tilde{u}, \tilde{v}; \Omega_{t})
=\tilde{\mathcal{J}}(s, \tilde{u}; \Omega_{t})\\
-\int_{\Omega\setminus\Sigma_{t}} \bigl( (C\circ \phi_s) E(\nabla\phi_s^{-1}\circ \phi_s,
\tilde{u})\cdot E(\nabla\phi_s^{-1}\circ \phi_s, \tilde{v}) \bigr) \,\omega^{\rm d}_s dx
+\int_{\Gamma^{\rm N}_{t}} (g\circ \phi_s)\cdot \tilde{v} \,\omega^{\rm b}_s dS_x\\
-\int_{\Sigma_t} \Bigl\{ \Bigl( \int_0^1 \nabla^2\alpha_{\rm f}
([\![r u^\varepsilon_t]\!]_{\tau_t}) [\![\tilde{u}]\!]_{\tilde{\tau}_{t+s}} \,dr
+\nabla\alpha_{\rm f}(0) \Bigr) \cdot [\![\tilde{v}]\!]_{\tilde{\tau}_{t+s}}\\
+\Bigl( \int_0^1 [\alpha_{\rm c}^\prime +\beta_\varepsilon]
(\nu_t\cdot [\![r u^\varepsilon_t]\!]) (\tilde{\nu}_{t+s} \cdot[\![\tilde{u}]\!]) \,dr
+[\alpha_{\rm c}^\prime +\beta_\varepsilon](0) \Bigr)
(\tilde{\nu}_{t+s} \cdot[\![\tilde{v}]\!]) \Bigr\} \,\omega^{\rm b}_s dS_x.
\end{multline}
In \eqref{4.10}, the following decomposition at $\Sigma_t$ was used
in accordance with \eqref{2.1}:
\begin{equation}\label{4.11}
[\![\tilde{u}]\!]_{\tilde{\tau}_{t+s}} :=[\![\tilde{u}]\!]
-(\tilde{\nu}_{t+s} \cdot[\![\tilde{u}]\!])\, \tilde{\nu}_{t+s},
\quad \tilde{\nu}_{t+s} :=\nu_{t+s}\circ \phi_s.
\end{equation}
Further in view of the chain rule $\nabla_y u =(\nabla\phi_s^{-T}\circ \phi_s)
\nabla (u\circ\phi_s)$, there appears the expression
\begin{equation}\label{4.12}
E(M, \tilde{u}) :=\frac{1}{2} (M^\top \nabla \tilde{u} +\nabla \tilde{u}^\top M),
\quad M\in \mathbb{R}^{d\times d},
\end{equation}
for which $E(I,\tilde{u}) =\epsilon(\tilde{u})$ according to \eqref{2.7}. 
For more details of the derivation, see \cite{Kov/06,KK/07,KO/20}.

\begin{lemma}[T2]\label{lem2}
The asymptotic expansion in the first argument of $\tilde{\mathcal{J}}$
from \eqref{4.9} is given by
\begin{equation}\label{4.14}
\tilde{\mathcal{J}}(s, \tilde{u}; \Omega_{t})
=\mathcal{J}(\tilde{u}; \Omega_{t}) +{\rm O}(|s|),
\end{equation}
and the expansion of $\tilde{\mathcal{L}}^\varepsilon$ from \eqref{4.10} by:
\begin{equation}\label{4.15}
\tilde{\mathcal{L}}^\varepsilon(s, u^\varepsilon_t, \tilde{u}, \tilde{v}; \Omega_{t})
=\tilde{\mathcal{L}}^\varepsilon(0, u^\varepsilon_t, \tilde{u}, \tilde{v}; \Omega_{t})
+s {\textstyle\frac{\partial}{\partial s}} \tilde{\mathcal{L}}^\varepsilon
(0, u^\varepsilon_t, \tilde{u}, \tilde{v}; \Omega_{t}) +{\rm o}(|s|)
\end{equation}
holds as $s\to0$.
The partial derivative $(t_0,t_1)-t\mapsto\mathbb{R}, \tau\mapsto
{\textstyle\frac{\partial}{\partial s}} \tilde{\mathcal{L}}^\varepsilon$
in \eqref{4.15} is a continuous function and exhibits the explicit representation
\begin{multline}\label{4.16}
{\textstyle\frac{\partial}{\partial s}} \tilde{\mathcal{L}}^\varepsilon
(\tau, u^\varepsilon_t, \tilde{u}, \tilde{v}; \Omega_{t})
=\int_{\Gamma^{\rm O}_{t}} \Bigl( \frac{1}{2}
{\rm div}_{\tau_t} \Lambda|_{t+\tau} |\tilde{u} -z|^2
-\nabla z \Lambda|_{t+\tau}\cdot (\tilde{u} -z) \Bigr) dS_x
+\rho \int_{\Sigma_{t}} {\rm div}_{\tau_t} \Lambda|_{t+\tau} dS_x\\
-\int_{\Omega\setminus\Sigma_{t}} \Bigl( \bigl( {\rm div} \Lambda|_{t+\tau} C
+\nabla C \Lambda|_{t+\tau} \bigr) \epsilon(\tilde{u})\cdot \epsilon(\tilde{v})
-\sigma(\tilde{u})\cdot E(\nabla \Lambda|_{t+\tau}, \tilde{v})
-\sigma(\tilde{v})\cdot E(\nabla \Lambda|_{t+\tau}, \tilde{u}) \Bigr) dx\\
+\int_{\Gamma^{\rm N}_{t}} \bigl( {\rm div}_{\tau_t} \Lambda|_{t+\tau} g
+\nabla g\, \Lambda|_{t+\tau} \bigr)\cdot \tilde{v} \,dS_x
-\int_{\Sigma_t} \Bigl\{
\int_0^1 \bigl( \nabla^2 \alpha_{\rm f}([\![r u^\varepsilon_t]\!]_{\tau_t})\,
[\![\tilde{u}]\!]_{\nabla \tau_t \Lambda|_{t+\tau}} \bigr)\cdot [\![\tilde{v}]\!]_{\tau_t} dr\\
+\Bigl( \int_0^1 \nabla^2 \alpha_{\rm f}
([\![r u^\varepsilon_t]\!]_{\tau_t})\, [\![\tilde{u}]\!]_{\tau_t} dr
+\nabla \alpha_{\rm f}(0) \Bigr)
\cdot \bigl( {\rm div}_{\tau_t} \Lambda|_{t+\tau} [\![\tilde{v}]\!]_{\tau_t}
+[\![\tilde{v}]\!]_{\nabla \tau_t \Lambda|_{t+\tau}} \bigr)\\
+\Bigl( \int_0^1 [\alpha_{\rm c}^{\prime\prime} +\beta^\prime_\varepsilon]
(\nu_t \cdot[\![r u^\varepsilon_t]\!])\, (\nu_t\cdot [\![\tilde{u}]\!]) \,dr
+[\alpha_{\rm c}^\prime +\beta_\varepsilon](0) \Bigr)
\bigl( ({\rm div}_{\tau_t} \Lambda|_{t+\tau} \nu_t
+\nabla \nu_t \Lambda|_{t+\tau})\cdot [\![\tilde{v}]\!] \bigr)\\
+\int_0^1 [\alpha_{\rm c}^{\prime\prime} +\beta^\prime_\varepsilon]
(\nu_t \cdot[\![r u^\varepsilon_t]\!])\, \bigl( \nabla\nu_t \Lambda|_{t+\tau}\cdot
[\![\tilde{u}]\!] \bigr) (\nu_t \cdot[\![\tilde{v}]\!]) \,dr \Bigr\} \,dS_x.
\end{multline}
In \eqref{4.16} the notation $\nabla \tau_t \Lambda$
and $\nabla \nu_t \Lambda$ at $\Sigma_t$ stands for
\begin{equation}\label{4.17}
[\![\tilde{u}]\!]_{\nabla \tau_t \Lambda}
:=-(\nu_t\cdot [\![\tilde{u}]\!]) \nabla \nu_t \Lambda
-(\nabla \nu_t \Lambda \cdot [\![\tilde{u}]\!]) \nu_t,\quad
\nabla \nu_t \Lambda :=\bigl( (\nabla \Lambda\, \nu_{t})\cdot
\nu_{t} \bigr) \nu_{t} -\nabla \Lambda^\top \nu_{t},
\end{equation}
and the tangential divergence is defined as
\begin{equation}\label{4.18}
{\rm div}_{\tau_t} \Lambda ={\rm div} \Lambda -(\nabla \Lambda\,
n^\pm_{t})\cdot n^\pm_{t}\text{ at }\partial\Omega^\pm_t.
\end{equation}
\end{lemma}

The proof of Lemma~\ref{lem2} is presented in Appendix~\ref{A}.

\begin{lemma}[T3]\label{lem1}
The set of saddle points $(\tilde{u}^\varepsilon_{t+s}, \tilde{v}^\varepsilon_{t+s})$
for \eqref{4.8} is a singleton for all $s\in[t_0, t_1] -t$, and $(\tilde{u}^\varepsilon_t,
\tilde{v}^\varepsilon_t) =(u^\varepsilon_t, v^\varepsilon_t)$ as $s=0$.
\end{lemma}

The proof is given in in Appendix~\ref{D} and
follows the arguments in the proof of Theorem~\ref{theo3}, which treats
a particular case of the saddle-point problem \eqref{4.8} as $s=0$.

\begin{lemma}[T4]\label{lem3}
There exists a subsequence $s_k\to0$ as $k\to\infty$, such that
\begin{equation}\label{4.19}
(\tilde{u}^\varepsilon_{t+s_k}, \tilde{v}^\varepsilon_{t+s_k})\to
(u^\varepsilon_t, v^\varepsilon_t)\quad\text{strongly in $V(\Omega_{t})^2$
as $s_k\to0$}.
\end{equation}
\end{lemma}

The proof of Lemma~\ref{lem3} is technical. It is presented in Appendix~\ref{B}.

Based on the  properties (T1)--(T4) we establish the main result of this section.

\begin{theorem}[Shape differentiability of $\varepsilon$-regularized
optimization problem]\label{theo4}
Under assumption \eqref{3.14}, the shape derivative (see its definition
\eqref{1.10} and existence criterion \eqref{1.12}) can be expressed
by the partial derivative from \eqref{4.16} as
\begin{multline}\label{4.20}
\partial_t \mathcal{J}(u^\varepsilon_t; \Omega_t)
={\textstyle\frac{\partial}{\partial s}} \tilde{\mathcal{L}}^\varepsilon
(0, u^\varepsilon_t, u^\varepsilon_t, v^\varepsilon_t; \Omega_{t})
=\int_{\Gamma^{\rm O}_{t}} \Bigl( \frac{1}{2}
{\rm div}_{\tau_t} \Lambda\, |u^\varepsilon_t -z|^2
-\nabla z \Lambda\cdot (u^\varepsilon_t -z) \Bigr) dS_x\\
-\int_{\Omega\setminus\Sigma_{t}} \Bigl( \bigl( {\rm div} \Lambda\, C
+\nabla C \Lambda \bigr) \epsilon(u^\varepsilon_t)\cdot \epsilon(v^\varepsilon_t)
-\sigma(u^\varepsilon_t)\cdot E(\nabla \Lambda, v^\varepsilon_t)
-\sigma(v^\varepsilon_t)\cdot E(\nabla \Lambda, u^\varepsilon_t) \Bigr) dx\\
+\int_{\Gamma^{\rm N}_{t}} \bigl( {\rm div}_{\tau_t} \Lambda\, g
+\nabla g\, \Lambda \bigr)\cdot v^\varepsilon_t \,dS_x
-\int_{\Sigma_t} \Bigl\{ \nabla \alpha_{\rm f}([\![u^\varepsilon_t]\!]_{\tau_t})
\cdot \bigl( {\rm div}_{\tau_t} \Lambda\, [\![v^\varepsilon_t]\!]_{\tau_t}
+[\![v^\varepsilon_t]\!]_{\nabla \tau_t \Lambda} \bigr)\\
+\int_0^1 \bigl( \nabla^2 \alpha_{\rm f}([\![r u^\varepsilon_t]\!]_{\tau_t})\,
[\![u^\varepsilon_t]\!]_{\nabla \tau_t \Lambda} \bigr)\cdot
[\![v^\varepsilon_t]\!]_{\tau_t} dr +[\alpha_{\rm c}^\prime +\beta_\varepsilon]
(\nu_t \cdot[\![u^\varepsilon_t]\!]) \bigl( ({\rm div}_{\tau_t} \Lambda\, \nu_t
+\nabla \nu_t \Lambda)\cdot [\![v^\varepsilon_t]\!] \bigr)\\
+\int_0^1 [\alpha_{\rm c}^{\prime\prime} +\beta^\prime_\varepsilon]
(\nu_t \cdot[\![r u^\varepsilon_t]\!])\, \bigl( \nabla\nu_t \Lambda\cdot
[\![u^\varepsilon_t]\!] \bigr) (\nu_t \cdot[\![v^\varepsilon_t]\!]) \,dr \Bigr\} \,dS_x
+\rho \int_{\Sigma_{t}} {\rm div}_{\tau_t} \Lambda \,dS_x,
\end{multline}
where $(u^\varepsilon_t, v^\varepsilon_t)\in V(\Omega_t)^2$
is a saddle-point to \eqref{3.11}.
\end{theorem}

\begin{proof}
Indeed, due to  (T1)--(T4)  all assumptions in Delfour--Zolesio
\cite[Chapter~10, Theorem~5.1]{DZ/11} are satisfied.
Details of the proof can be found in \cite{KO/20}.
\qed
\end{proof}

\begin{corollary}[Hadamard representation of the $\varepsilon$-dependent
shape derivative]\label{corol1}
Assume that the solution of \eqref{3.4} and \eqref{3.15} satisfies
$(u^\varepsilon_t, v^\varepsilon_t)\in H^2(\Omega^+_t)^{2d}\cap
H^2(\Omega^-_t)^{2d}$.
Introducing the decomposition into  normal and tangential components according to
\begin{equation}\label{4.21}
\Lambda = (n_t\cdot \Lambda) n_t +\Lambda_{\tau_t},\quad
\nabla = (n_t\cdot \nabla) n_t +\nabla_{\tau_t},\quad
\mathcal{D} = (n_t\cdot \mathcal{D}) n_t +\mathcal{D}_{\tau_t},
\end{equation}
the following equivalent representation of the shape derivative \eqref{4.20}
holds in terms of boundary integrals in 2D:
\begin{multline}\label{4.22}
{\textstyle\frac{\partial}{\partial s}} \tilde{\mathcal{L}}^\varepsilon
(0, u^\varepsilon_t, u^\varepsilon_t, v^\varepsilon_t; \Omega_{t})\\
=\int_{\Gamma^{\rm D}_{t}} (\tau_t\cdot\Lambda) \tau_t\cdot \mathcal{D}_1
(u^\varepsilon_t, v^\varepsilon_t) \,dS_x
+\int_{\Sigma_{t}} \bigl( (\tau_t\cdot\Lambda) \tau_t\cdot
\mathcal{D}^\varepsilon_2 (u^\varepsilon_t, v^\varepsilon_t)
+(\nu_t\cdot \Lambda) \mathcal{D}^\varepsilon_3
(u^\varepsilon_t, v^\varepsilon_t) \bigr) \,dS_x\\
+(\tau_t\cdot\Lambda) [\![\mathcal{D}^\varepsilon_4 (u^\varepsilon_t,
v^\varepsilon_t)]\!]_{\partial\Sigma_{t}}
+(\tau_t\cdot\Lambda) \mathcal{D}_5
(u^\varepsilon_t)|_{\partial\Gamma^{\rm O}_{t}}
+(\tau_t\cdot\Lambda) [\![\mathcal{D}_6
(v^\varepsilon_t)]\!]_{\partial\Gamma^{\rm N}_{t}\cap\Sigma_{t}},
\end{multline}
where $\tau_t$ is a tangential vector at the boundary, and in 3D:
\begin{multline}\label{4.23}
=\int_{\Gamma^{\rm D}_{t}} \Lambda_{\tau_t}\cdot \mathcal{D}_1
(u^\varepsilon_t, v^\varepsilon_t)_{\tau_t} dS_x
+\int_{\Sigma_{t}} \bigl( \Lambda_{\tau_t}\cdot
\mathcal{D}^\varepsilon_2 (u^\varepsilon_t, v^\varepsilon_t)_{\tau_t}
+(\nu_t\cdot \Lambda) \mathcal{D}^\varepsilon_3
(u^\varepsilon_t, v^\varepsilon_t) \bigr) \,dS_x\\
+{\displaystyle\int_{\partial\Sigma_{t}}} \!\!\!\! (b_t\cdot\Lambda)
[\![\mathcal{D}^\varepsilon_4 (u^\varepsilon_t, v^\varepsilon_t)]\!] dL_x
+{\displaystyle\int_{\partial\Gamma^{\rm O}_{t}}} \!\!\!\!
(b_t\cdot\Lambda) \mathcal{D}_5 (u^\varepsilon_t) dL_x
+{\displaystyle\int_{\partial\Gamma^{\rm N}_{t}\cap\Sigma_{t}}} \!\!\!\!
(b_t\cdot\Lambda) [\![\mathcal{D}_6 (v^\varepsilon_t)]\!] dL_x,
\end{multline}
where $b_t =\tau_t\times n_t$ is a binomial vector within the moving frame
at the respective boundary.
The terms in \eqref{4.22} and \eqref{4.23} are
\begin{multline}\label{4.24}
\mathcal{D}_1(\tilde{u}, \tilde{v}) :=\nabla \tilde{u}^\top \sigma(\tilde{v}) n_t
+\nabla \tilde{v}^\top \sigma(\tilde{u}) n_t,\quad
\mathcal{D}^\varepsilon_2(\tilde{u}, \tilde{v})
:=-[q_{\rm f} +q^\varepsilon_{\rm c}]_{\tau_t} (\tilde{u}, \tilde{v}),\\
\mathcal{D}^\varepsilon_3(\tilde{u}, \tilde{v})
:=[\![\sigma(\tilde{u})\cdot \epsilon(\tilde{v})]\!] +\rho \varkappa_{t}
-\varkappa_{t}  [p_{\rm f} +p^\varepsilon_{\rm c}] (\tilde{u}, \tilde{v})
-\nu_t\cdot [\nabla  (p_{\rm f} +p^\varepsilon_{\rm c})
+q_{\rm f} +q^\varepsilon_{\rm c}] (\tilde{u}, \tilde{v}),\\
\mathcal{D}^\varepsilon_4(\tilde{u}, \tilde{v})
:=\rho -[p_{\rm f} +p^\varepsilon_{\rm c}] (\tilde{u}, \tilde{v}),\quad
\mathcal{D}_5(\tilde{u}) :=\frac{1}{2} |\tilde{u} - z|^2,\quad
\mathcal{D}_6(\tilde{v}) :=g\cdot\tilde{v},
\end{multline}
with the curvature $\varkappa_{t} ={\rm div}_{\tau_t} \nu_{t}$ at $\Sigma_t$.
The expressions along $\Sigma_{t}$ are defined by
\begin{equation}\label{4.25}
p_{\rm f}(\tilde{u}, \tilde{v}) :=\nabla\alpha_{\rm f}
([\![\tilde{u}]\!]_{\tau_t}) \cdot[\![\tilde{v}]\!]_{\tau_t},\quad
p^\varepsilon_{\rm c}(\tilde{u}, \tilde{v})
:=[\alpha_{\rm c}^\prime +\beta_\varepsilon] (\nu_t\cdot[\![\tilde{u}]\!])\,
(\nu_t\cdot[\![\tilde{v}]\!]),
\end{equation}
and next
\begin{multline}\label{4.26}
q_{\rm f}(\tilde{u}, \tilde{v}) :=[\![\nabla \tilde{v}]\!]^\top\nu_t
\bigl( \nu_t\cdot \nabla\alpha_{\rm f} ([\![\tilde{u}]\!]_{\tau_t}) \bigr)
+[\![\nabla \tilde{u}]\!]^\top\nu_t \Bigl( \nu_t\cdot \int_0^1 \nabla^2\alpha_{\rm f}
([\![r u^t_\varepsilon]\!]_{\tau_t}) [\![\tilde{v}]\!]_{\tau_t} dr \Bigr)\\
+\nabla ([\![\tilde{u}]\!]_{\tau_t})^\top \int_0^1 \bigl(
\nabla^2 \alpha_{\rm f} ([\![r u^\varepsilon_t]\!]_{\tau_t})
-\nabla^2 \alpha_{\rm f} ([\![u^\varepsilon_t]\!]_{\tau_t}) \bigr)
[\![\tilde{v}]\!]_{\tau_t} \,dr,\\
q^\varepsilon_{\rm c}(\tilde{u}, \tilde{v})
:=\nabla(\nu_t\cdot [\![\tilde{u}]\!])^\top \int_0^1 \bigl(
[\alpha^{\prime\prime}_{\rm c} +\beta^\prime_\varepsilon]
(\nu_t\cdot [\![r u^\varepsilon_t]\!]) -[\alpha^{\prime\prime}_{\rm c}
+\beta^\prime_\varepsilon] (\nu_t\cdot [\![u^\varepsilon_t]\!]) \bigr)
 (\nu_t\cdot [\![\tilde{v}]\!]) \,dr.
\end{multline}
\end{corollary}

The proof of Corollary~\ref{corol1} is given in Appendix~\ref{C}.

We remark that the additional $H^2$-regularity is available when a piecewise
$C^{2,0}$-boundaries $\partial\Omega^\pm_t$ exclude singular points
(e.g. in 2D when the boundary parts meet each other with an $\pi/2$-angle
as in Figure~\ref{fig1}).

\begin{corollary}[Descent direction for the $\varepsilon$-dependent optimization]\label{corol2}
A descent direction for the perturbed $\tilde{\mathcal{L}}^\varepsilon$ in \eqref{4.15}
is provided by the following choice of the velocity
\begin{multline}\label{4.27}
\tau_t\cdot\Lambda =-k_1 \tau_t\cdot\mathcal{D}_1 (u^\varepsilon_t, v^\varepsilon_t)
\text{ at $\Gamma^{\rm D}_{t}$},\quad
\tau_t\cdot\Lambda =-k_2 \tau_t\cdot \mathcal{D}^\varepsilon_2
(u^\varepsilon_t, v^\varepsilon_t)\text{ and }
\nu_t\cdot \Lambda =-k_3 \mathcal{D}^\varepsilon_3
(u^\varepsilon_t, v^\varepsilon_t)\text{ at $\Sigma_{t}$},\\
\tau_t\cdot\Lambda =-k_4 [\![\mathcal{D}^\varepsilon_4
(u^\varepsilon_t, v^\varepsilon_t)]\!]\text{ at $\partial\Sigma_{t}$},\;
\tau_t\cdot\Lambda =-k_5\mathcal{D}_5 (u^\varepsilon_t)
\text{ at $\partial\Gamma^{\rm O}_{t}$},\;
\tau_t\cdot\Lambda =-k_6 [\![\mathcal{D}_6 (v^\varepsilon_t)]\!]
\text{at $\partial\Gamma^{\rm N}_{t}\cap\Sigma_{t}$},\\
n_t\cdot\Lambda =0\text{ at $\partial\Omega$}
\end{multline}
in 2D, and in 3D respectively
\begin{multline}\label{4.28}
\Lambda_{\tau_t} =-k_1 \mathcal{D}_1 (u^\varepsilon_t, v^\varepsilon_t)_{\tau_t}
\text{ at $\Gamma^{\rm D}_{t}$},\quad
\Lambda_{\tau_t} =-k_2 \mathcal{D}^\varepsilon_2
(u^\varepsilon_t, v^\varepsilon_t)_{\tau_t}\text{and }
\nu_t\cdot \Lambda =-k_3 \mathcal{D}^\varepsilon_3
(u^\varepsilon_t, v^\varepsilon_t)\text{ at $\Sigma_{t}$},\\
b_t\cdot\Lambda =-k_4 [\![\mathcal{D}^\varepsilon_4
(u^\varepsilon_t, v^\varepsilon_t)]\!]\text{ at $\partial\Sigma_{t}$},\;
b_t\cdot\Lambda =-k_5\mathcal{D}_5 (u^\varepsilon_t)
\text{ at $\partial\Gamma^{\rm O}_{t}$},\;
b_t\cdot\Lambda =-k_6 [\![\mathcal{D}_6 (v^\varepsilon_t)]\!]
\text{at $\partial\Gamma^{\rm N}_{t}\cap\Sigma_{t}$},\\
n_t\cdot\Lambda =0\text{ at $\partial\Omega$},
\end{multline}
with $k_i\ge0$, $i=1,\ldots,6$, and not all simultaneously equal to zero.
\end{corollary}

\begin{proof}
Direct substitution of \eqref{4.27} into \eqref{4.22} in 2D, respectively
\eqref{4.28} into \eqref{4.23} in 3D, provides that
${\textstyle\frac{\partial}{\partial s}} \tilde{\mathcal{L}}^\varepsilon
(0, u^\varepsilon_t, u^\varepsilon_t, v^\varepsilon_t; \Omega_{t}) <0$.
\qed
\end{proof}

Corollary~\ref{corol2} is of practical importance since it provides well-posedness of
gradient schemes (see Algorithm~\ref{algo1}) based on the descent direction from \eqref{4.27} and  \eqref{4.28}.

\section{The limit as $\varepsilon\to0^+$}\label{sec5}

In the following we derive the limit relations as $\varepsilon\to0^+$.
We recall that all results involving the dual variable
$v^\varepsilon_t$ assume that \eqref{3.14} holds true.

\begin{lemma}[Uniform estimate]\label{lem4}
The following a-priori estimate holds uniformly in $\varepsilon\in (0, \varepsilon_0)$:
\begin{equation}\label{5.1}
\|u^\varepsilon_t\|_{H^1(\Omega\setminus\Sigma_t)^d} +\frac{1}{\sqrt{\varepsilon}}
\| \bigl[ \nu_t \cdot[\![u^\varepsilon_t]\!] \bigr]^-\|_{L^2(\Sigma_t)}
+\|v^\varepsilon_t\|_{H^1(\Omega\setminus\Sigma_t)^d}
\le K,\quad K\ge0.
\end{equation}
Consequently, there exists a subsequence $\varepsilon_k\to0$ as $k\to\infty$
and an accumulation point $(u_t, v_t)\in  K(\Omega_t)\times  V(\Omega_t)$ such that
\begin{equation}\label{5.2}
(u^{\varepsilon_k}_t, v^{\varepsilon_k}_t)\to(u_t, v_t)\text{ weakly in
$H^1(\Omega\setminus\Sigma_t)^{2d}$, $H^{1/2}(\partial\Omega^\pm_t)^{2d}$,
strongly in $L^2(\partial\Omega^\pm_t)^{2d}$}.
\end{equation}
\end{lemma}

\begin{proof}
Passing $s\to0$ due to the convergences \eqref{B8} and \eqref{B9} and using
the lower bound $\beta_\epsilon(\nu_t \cdot[\![u^\varepsilon_t]\!])
(\nu_t \cdot[\![u^\varepsilon_t]\!])\ge ([ \nu_t \cdot [\![u^\varepsilon_t]\!] ]^-)^2/
\varepsilon -2\varepsilon K_{\beta}$ due to \eqref{3.2}, in the limit
we improve the uniform a-priori estimate \eqref{B7} and get \eqref{5.1}.
Consequently \eqref{5.2} follows by  a standard compactness argument.
Moreover, $[\nu_t \cdot[\![u^{\varepsilon_k}_t]\!]]^-\to0$ ensures
$\nu_t \cdot[\![u_t]\!]\ge0$ at $\Sigma_t$, hence $u_t\in  K(\Omega_t)$.
\qed
\end{proof}

Let $u_t\in  K(\Omega_t)$ be a solution to the VI \eqref{2.13} in Theorem~\ref{theo1}.
According to \eqref{3.10} we introduce the $\varepsilon$-independent Lagrangian
$(u, v)\mapsto \mathcal{L}: V(\Omega_t)^2\mapsto\mathbb{R}$ as
\begin{multline}\label{5.3}
\mathcal{L}(u_t, u, v; \Omega_t) :=\frac{1}{2} \int_{\Gamma^{\rm O}_t} |u - z|^2 \,dS_x
+\rho |\Sigma_t| -\int_{\Omega\setminus\Sigma_t} \sigma(u)\cdot \epsilon(v) \,dx
+\int_{\Gamma^{\rm N}_t} g\cdot v \,dS_x\\ -\int_{\Sigma_t} \Bigl\{
\Bigl( \int_0^1 \nabla^2 \alpha_{\rm f} ([\![r u_t]\!]_{\tau_t})  \,
[\![u]\!]_{\tau_t} dr +\nabla\alpha_{\rm f}(0) \Bigr) \cdot [\![v]\!]_{\tau_t}\\
+\Bigl( \int_0^1 \alpha_{\rm c}^{\prime\prime} (\nu_t\cdot[\![r u_t]\!]) \,
(\nu_t\cdot[\![u]\!]) \,dr +\alpha_{\rm c}^\prime(0) \Bigr) (\nu_t\cdot[\![v]\!]) \Bigr\} \,dS_x.
\end{multline}
Based on Lemma~\ref{lem4} we prove the following.

\begin{theorem}[Limit optimality conditions]\label{theo5}
(i)
There exists a pair $(u_t, \lambda_t)\in V(\Omega_t) \times H^{1/2}(\Sigma_t)^\star$
which satisfies the variational equation
\begin{multline}\label{5.4}
\int_{\Omega\setminus\Sigma_t} \sigma(u_t)\cdot \epsilon(u) \,dx
+\int_{\Sigma_t} \bigl\{ \nabla\alpha_{\rm f}([\![u_t]\!]_{\tau_t})\cdot [\![u]\!]_{\tau_t}
+\alpha_{\rm c}^\prime(\nu_t\cdot[\![u_t]\!])\, (\nu_t\cdot[\![u]\!])\bigr\} \,dS_x\\
+\langle \lambda_t, \nu_t\cdot[\![u]\!] \rangle_{\Sigma_t}
=\int_{\Gamma^{\rm N}_t} g\cdot u \,dS_x
\end{multline}
for all test functions $u\in V(\Omega_t)$, simultaneously with the complementary relations
\begin{equation}\label{5.5}
\nu_t\cdot[\![u_t]\!]\ge0,\quad \lambda_t\le0,\quad
\langle \lambda_t, \nu_t\cdot[\![u_t]\!] \rangle_{\Sigma_t} =0,
\end{equation}
where $\langle\,\cdot\,, \,\cdot\,\rangle_{\Sigma_t}$ stands for the duality pairing
between $H^{1/2}(\Sigma_t)$ and its dual space $H^{1/2}(\Sigma_t)^\star$.
The first component $u_t\in K(\Omega_t)$ solves the VI \eqref{2.13},
and according to \eqref{2.14} the second, $\lambda_t$, satisfies
\begin{equation}\label{5.6}
\lambda_t =\nu_t\cdot(\sigma(u_t)\nu_t) -\alpha_{\rm c}^\prime(\nu_t\cdot[\![u_t]\!])
\quad\text{at $\Sigma_t$}.
\end{equation}

(ii)
Under the assumption \eqref{3.14}, an adjoint pair $(v_t, \mu_t)\in V(\Omega_t)
\times H^{1/2}(\Sigma_t)^\star$ exists and satisfies the adjoint equation
\begin{multline}\label{5.7}
\int_{\Omega\setminus\Sigma_t} \sigma(v)\cdot \epsilon(v_t) \,dx
+\int_{\Sigma_t} \int_0^1 \Bigl\{ \Bigl( \nabla^2 \alpha_{\rm f}
([\![r u_t]\!]_{\tau_t})  \, [\![v]\!]_{\tau_t} \Bigr)\cdot [\![v_t]\!]_{\tau_t}\\
+\alpha_{\rm c}^{\prime\prime} (\nu_t\cdot[\![r u_t]\!]) \,(\nu_t\cdot[\![v]\!])
(\nu_t\cdot[\![v_t]\!]) \Bigr\} \,dr \,dS_x
+\langle \mu_t, \nu_t\cdot[\![v]\!] \rangle_{\Sigma_t}
=\int_{\Gamma^{\rm O}_t} (u_t -z)\cdot v \,dS_x
\end{multline}
for all test functions $v\in V(\Omega_t)$, such that the compatibility relation holds:
\begin{equation}\label{5.8}
\langle \lambda_t -\beta_{\varepsilon}(0), \nu_t\cdot[\![v_t]\!] \rangle_{\Sigma_t}
=\langle \mu_t, \nu_t\cdot[\![u_t]\!] \rangle_{\Sigma_t},
\end{equation}
where $\beta_{\varepsilon}(0) =-exp(-2)$ in \eqref{3.3} does not depend on $\varepsilon$.
In case $v_t$ is smooth, the following boundary value relations hold:
\begin{align}\label{5.9}
{\rm div}\, \sigma(v_t) =0 &\text{ in } \Omega\setminus\Sigma_t,\nonumber\\
v_t =0 \text{ on } \Gamma^{\rm D}_t,\quad \sigma(v_t)n =u_t -z
\text{ on } \Gamma^{\rm O}_t,\quad \sigma(v_t)n =0 &\text{ on }
\Gamma^{\rm N}_t\setminus \Gamma^{\rm O}_t,\nonumber\\
[\![\sigma(v_t)\nu_t]\!] =0,\quad (\sigma(v_t)\nu_t)_{\tau_t}
=\int_0^1 \nabla^2 \alpha_{\rm f} ([\![r u_t]\!]_{\tau_t})
\, [\![v_t]\!]_{\tau_t} \,dr,&\nonumber\\
\nu_t\cdot(\sigma(v_t)\nu_t) =\int_0^1 \alpha_{\rm c}^{\prime\prime}
(\nu_t\cdot[\![r u_t]\!]) \,(\nu_t\cdot[\![v_t]\!]) \,dr +\mu_t
&\text{ on } \Sigma_t.
\end{align}

(iii)
The quadruple $(u_t, v_t, \lambda_t, \mu_t)$ constitutes an accumulation point
as $\varepsilon_k\to0$:
\begin{equation}\label{5.10}
u^{\varepsilon_k}_t\to u_t\text{ strongly in $H^1(\Omega\setminus\Sigma_t)^d$},
\quad v^{\varepsilon_k}_t\rightharpoonup v_t\text{ weakly in $H^1(\Omega\setminus\Sigma_t)^d$},
\end{equation}
\begin{equation}\label{5.11}
\beta_{\varepsilon_k}(\nu_t\cdot[\![u^{\varepsilon_k}_t]\!])\to \lambda_t
\text{ strongly in $H^{1/2}(\Sigma_t)^\star$},
\end{equation}
\begin{equation}\label{5.12}
\int_0^1 \beta^\prime_{\varepsilon_k}(\nu_t\cdot[\![r u^{\varepsilon_k}_t]\!])
\, (\nu_t\cdot[\![v^{\varepsilon_k}_t]\!]) \,dr \rightharpoonup \mu_t
\text{ $\star$-weakly in $H^{1/2}(\Sigma_t)^\star$}.
\end{equation}
\end{theorem}

\begin{proof}
(i)
Taking the limit in \eqref{3.4} with the help of the weak convergence
$u^{\varepsilon_k}_t\rightharpoonup u_t$ in \eqref{5.2} we get
\begin{multline}\label{5.13}
\lim_{\varepsilon_k\to0} \int_{\Sigma_t} \beta_{\varepsilon_k}
(\nu_t\cdot[\![u^{\varepsilon_k}_t]\!])\, (\nu_t\cdot[\![u]\!]) \,dS_x
=\int_{\Gamma^{\rm N}_t} g\cdot u \,dS_x
-\int_{\Omega\setminus\Sigma_t} \sigma(u_t)\cdot \epsilon(u) \,dx\\
-\int_{\Sigma_t} \bigl\{ \nabla\alpha_{\rm f}([\![u_t]\!]_{\tau_t})
\cdot [\![u]\!]_{\tau_t} +\alpha_{\rm c}^\prime (\nu_t\cdot[\![u_t]\!])\,
(\nu_t\cdot[\![u]\!])\bigr\} \,dS_x =:\langle \lambda_t, \nu_t\cdot[\![u]\!] \rangle_{\Sigma_t}.
\end{multline}
This implies  the $\star$-weak convergence
$\beta_{\varepsilon_k}(\nu_t\cdot[\![u^{\varepsilon_k}_t]\!])\rightharpoonup \lambda_t$
in $H^{1/2}(\Sigma_t)^\star$, equation \eqref{5.4}, and $\lambda_t\le0$
in \eqref{5.5} due to $\beta_{\varepsilon_k}\le0$ in  Lemma~\ref{lem0}.
Testing \eqref{5.13} with $u =u^{\varepsilon_k}_t$ and using \eqref{3.2} such that
\begin{equation*}
\int_{\Sigma_t} \beta_{\varepsilon_k} (\nu_t\cdot[\![u^{\varepsilon_k}_t]\!])\,
(\nu_t\cdot[\![u^{\varepsilon_k}_t]\!]) \,dS_x \ge\frac{1}{\varepsilon_k} \int_{\Sigma_t}
(\bigl[\nu_t\cdot[\![u^{\varepsilon_k}_t]\!] \bigr]^-)^2 \,dS_x  -2 \varepsilon_k K_{\beta}
\ge -2 \varepsilon_k K_{\beta}\to0,
\end{equation*}
after passage $\varepsilon_k\to0$, we get in the limit
$\langle \lambda_t, \nu_t\cdot[\![u_t]\!] \rangle_{\Sigma_t}\ge0$.
On the other hand we have
$\langle \lambda_t, \nu_t\cdot[\![u_t]\!] \rangle_{\Sigma_t}\le0$
because $\lambda_t\le0$ and the non-penetration $\nu_t\cdot[\![u_t]\!]\ge0$,
which together lead to the equality in \eqref{5.5}.
Substituting $\lambda_t$ with the expression \eqref{5.6} at $\Sigma_t$
we derive the VI \eqref{2.13} and its boundary value formulation \eqref{2.14}.
Thus, $u_t\in K(\Omega_t)$ yields a solution of the cohesive crack problem.

(ii)
The limit of the adjoint equation \eqref{3.15} using the convergences in \eqref{5.2} is
\begin{multline}\label{5.14}
\lim_{\varepsilon_k\to0} \int_{\Sigma_t} \int_0^1 \beta^\prime_{\varepsilon_k}
(\nu_t\cdot[\![r u^{\varepsilon_k}_t]\!])\, (\nu_t\cdot[\![v]\!])\,
(\nu_t\cdot[\![v^{\varepsilon_k}_t]\!]) \,dr \,dS_x\\
=\int_{\Gamma^{\rm O}_t} (u_t -z)\cdot v \,dS_x
-\int_{\Omega\setminus\Sigma_t} \sigma(v)\cdot \epsilon(v_t) \,dx
-\int_{\Sigma_t} \int_0^1 \Bigl\{ \Bigl( \nabla^2 \alpha_{\rm f}
([\![r u_t]\!]_{\tau_t})\, [\![v]\!]_{\tau_t} \Bigr) \cdot [\![v_t]\!]_{\tau_t}\\
+\alpha_{\rm c}^{\prime\prime} (\nu_t\cdot[\![r u_t]\!])
\,(\nu_t\cdot[\![v]\!]) (\nu_t\cdot[\![v_t]\!])\Bigr\} \,dS_x
=:\langle \mu_t, \nu_t\cdot[\![v]\!] \rangle_{\Sigma_t}.
\end{multline}
The convergence in \eqref{5.14} implies \eqref{5.12} and the adjoint equation \eqref{5.7}.
Derivation of the boundary value relations \eqref{5.9} is standard.
According to \eqref{3.9} we have
\begin{equation*}
\langle \beta_{\varepsilon_k} (\nu_t\cdot[\![u^{\varepsilon_k}_t]\!]),
\nu_t\cdot[\![v^{\varepsilon_k}_t]\!] \rangle_{\Sigma_t}
=\langle \int_0^1 \beta^\prime_{\varepsilon_k} (\nu_t\cdot[\![r u^{\varepsilon_k}_t]\!])
\,(\nu_t\cdot[\![u^{\varepsilon_k}_t]\!]) \,dr
+\beta_\epsilon(0), \nu_t\cdot[\![v^{\varepsilon_k}_t]\!] \rangle_{\Sigma_t},
\end{equation*}
hence based on \eqref{5.11} and \eqref{5.12} we derive in the limit the compatibility
equation \eqref{5.8}.

(iii)
The weak convergences in \eqref{5.10} are proved in Lemma~\ref{lem4}.
To justify the strong convergence $u^{\varepsilon_k}_t -u_t\to0$,
we subtract \eqref{5.4} from \eqref{3.4}, test the difference with
$u =u^{\varepsilon_k}_t -u_t$ and rearrange the terms as follows
\begin{multline}\label{5.15}
\int_{\Omega\setminus\Sigma_t} \sigma(u^{\varepsilon_k}_t -u_t)\cdot
\epsilon(u^{\varepsilon_k}_t -u_t) \,dx =-\int_{\Sigma_t} \bigl\{
\bigl( \nabla\alpha_{\rm f}([\![u^{\varepsilon_k}_t]\!]_{\tau_t}) -\nabla\alpha_{\rm f}
([\![u_t]\!]_{\tau_t}) \bigr)\cdot [\![u^{\varepsilon_k}_t -u_t]\!]_{\tau_t}\\
+\bigl( [\alpha_{\rm c}^\prime +\beta_{\varepsilon_k}] (\nu_t\cdot[\![u^{\varepsilon_k}_t]\!])
-[\alpha_{\rm c}^\prime +\beta_{\varepsilon_k}] (\nu_t\cdot[\![u_t]\!]) \bigr)
(\nu_t\cdot[\![u^{\varepsilon_k}_t -u_t]\!])\bigr\} \,dS_x\\
-\langle \beta_{\varepsilon_k} (\nu_t\cdot[\![u_t]\!]) -\lambda_t,
\nu_t\cdot[\![u^{\varepsilon_k}_t -u_t]\!] \rangle_{\Sigma_t}.
\end{multline}
Using the monotony of $\beta_{\varepsilon_k}$ and the uniform boundedness
$-1<\beta_{\varepsilon_k}(0)\le \beta_{\varepsilon_k} (\nu_t\cdot[\![u_t]\!]) \le0$
for $\nu_t\cdot[\![u_t]\!]\ge0$, the strong convergence in \eqref{5.10} follows
upon taking the limit in \eqref{5.15} as $\varepsilon_k\to0$, see \eqref{5.2}.
Consequently, from \eqref{3.4} and \eqref{5.13} we conclude the strong convergence in \eqref{5.11}.
This finishes the proof.
\qed
\end{proof}

Based on assertion (iii) of Theorem~\ref{theo5} we get the following.

\begin{corollary}[Limit optimization problems]\label{corol3}
For the fixed $(\lambda_t, \mu_t)\in (H^{1/2}(\Sigma_t)^\star)^2$
from Theorem~\ref{theo5} and Lagrangian $\mathcal{L}$ from \eqref{5.3},
the pair $(u_t, v_t)\in V(\Omega_t)^2$ solving optimality conditions
\eqref{5.4}, \eqref{5.5} and \eqref{5.7} satisfies the primal problem:
\begin{equation}\label{5.16}
\mathcal{L}(u_t, u_t, v;\Omega_t) -\langle \lambda_t, \nu_t\cdot[\![v]\!] \rangle_{\Sigma_t}
\le \mathcal{L}(u_t, u_t, v_t;\Omega_t) -\langle \lambda_t,
\nu_t\cdot[\![v_t]\!] \rangle_{\Sigma_t}
\end{equation}
for all $v\in V(\Omega_t)$, and the dual problem:
\begin{multline}\label{5.17}
\mathcal{L}(u_t, u_t, v_t;\Omega_t) -\langle \mu_t, \nu_t\cdot[\![u_t]\!] \rangle_{\Sigma_t}
-\langle \beta_\epsilon(0), \nu_t\cdot[\![v_t]\!] \rangle_{\Sigma_t}\\
\le \mathcal{L}(u_t, u, v_t;\Omega_t) -\langle \mu_t, \nu_t\cdot[\![u]\!] \rangle_{\Sigma_t}
-\langle \beta_\epsilon(0), \nu_t\cdot[\![v_t]\!] \rangle_{\Sigma_t}
\quad\text{for all } u\in V(\Omega_t).
\end{multline}
By the virtue of compatibility \eqref{5.8}, the corresponding
optimal value function for the objective $\mathcal{J}$ in \eqref{1.3}
has the equivalent representations using the adjoint equation as follows:
\begin{multline}\label{5.18}
\mathcal{J}(u_t; \Omega_t) =\mathcal{L}(u_t, u_t, v_t;\Omega_t)
-\langle \lambda_t, \nu_t\cdot[\![v_t]\!] \rangle_{\Sigma_t}\\
=\mathcal{L}(u_t, u_t, v_t;\Omega_t)-\langle \mu_t, \nu_t\cdot[\![u_t]\!] \rangle_{\Sigma_t}
-\langle \beta_\epsilon(0), \nu_t\cdot[\![v_t]\!] \rangle_{\Sigma_t}.
\end{multline}
\end{corollary}

\begin{proof}
Indeed, taking the limit $\varepsilon_k\to0$ in the saddle-point problem \eqref{3.11}
with the Lagrangian $\mathcal{\tilde L}^{\varepsilon_k}$ from \eqref{3.13}, and
observing \eqref{5.10}--\eqref{5.12}, the inequalities \eqref{5.16}, \eqref{5.17} follow.
From the  $\varepsilon$-dependent representation \eqref{3.12}
of the optimal value function $\mathcal{J}$ and by using the compatibility \eqref{5.8}
we derive the limit formula \eqref{5.18}.
\qed
\end{proof}

We finish by noting the difficulty that, in general, we can pass to the limit
as $\varepsilon\to0^+$ neither in the term $\int_0^1 \beta^\prime_\varepsilon (\nu_t\cdot
[\![r u^\varepsilon_t]\!]) \,dr$ in the Lagrangian $\tilde{\mathcal{L}}^\varepsilon$
in \eqref{3.10}, nor in the term $\beta^\prime_\varepsilon (\nu_t\cdot[\![u^\varepsilon_t]\!])$
in the shape derivative ${\textstyle\frac{\partial}{\partial s}}
\tilde{\mathcal{L}}^\varepsilon$ in \eqref{4.20} and \eqref{4.26}.
Otherwise, if
\begin{equation*}
\eta_t =\lim_{\varepsilon\to0^+} \int_0^1
\beta^\prime_\varepsilon (\nu_t\cdot [\![r u^\varepsilon_t]\!]) \,dr
\end{equation*}
exists, then the compatibility properties
$\lambda_t =(\nu_t\cdot[\![u_t]\!]) \eta_t +\beta_{\varepsilon}(0)$
and $\mu_t =(\nu_t\cdot[\![v_t]\!]) \eta_t$ which are stronger than \eqref{5.8} hold.
For a factorization of $\lambda_t$ and $\mu_t$,  additional solution
regularity, as in the particular case of obstacle problems, could be helpful,
see \cite{Bar/84,HK/09,MP/84}.

\section{Shape optimization of breaking line}\label{sec6}

We apply the theoretical results to a numerical example in 2D.

As a true shape to be identified within an admissibility set $\{\Sigma_t\}$
we take the piecewise-linear line
\begin{equation*}
\Sigma = \{x_1\in(0,1),\, x_2 = \psi(x_1) \},\quad \psi(x_1) =\min(0.3, x_1/3+0.1),
\end{equation*}
which breaks the rectangle $\Omega=(0,1)\times(0,0.5)$ into two parts $\Omega^\pm$.
Let the boundary $\partial \Omega$ be split symmetrically into the fixed Dirichlet
part $\Gamma^{\rm D} =\{x_1\in\{0,1\},\, x_2 \in(0,0.5) \}$ and
the Neumann part $\Gamma^{\rm N} =\{x_1 \in(0,1),\, x_2\in\{0,0.5\} \}$.
For an isotropic elastic body occupying $\Omega$ we set the material parameters:
Young modulus $E_{\rm Y} = 73000$ (mPa), Poisson ratio $\nu_{\rm P} = 0.34$,
and the corresponding Lam\'e parameters $\mu_{\rm L} = E_{\rm Y}/ (2(1+\nu_{\rm P}))$,
$\lambda_{\rm L} = 2 \mu_{\rm L} \nu_{\rm P}/ (1-2 \nu_{\rm P})$.
For the matrix $C$ of isotropic elastic coefficients the stress-strain relations are
\begin{equation*}
\sigma_{ij} =2\mu_{\rm L} \epsilon_{ij} +\lambda_{\rm L}
(\epsilon_{11} +\epsilon_{22}) \delta_{ij},\quad i,j=1,2.
\end{equation*}
We rely on the approximation of $\nu_t\cdot[\![u]\!]$ by $[\![u]\!]_2 :=[\![u_2]\!]$,
and $[\![u]\!]_{\tau_t} =[\![u]\!]_1 \tau_t$ with $[\![u]\!]_1 :=[\![u_1]\!]$
at $\Sigma_t$, which is reasonable for flat shapes.
For a friction function in one variable $\alpha_{\rm f}(s) =F_{\bf b} \sqrt{\delta^2 +s^2}$
such that $\nabla \alpha_{\rm f} = \tau_t \alpha_{\rm f}^\prime$,
and $\alpha_{\rm c}(s) =K_{\rm c} s/ (\kappa +|s|)$,
applying to the body the traction force
\begin{equation*}
g_1 = 0,\quad g_2(x) = (1-7 x_1/4)(4 x_2-1) \mu_{\rm L},
\end{equation*}
according to Theorem~\ref{theo1} there exists a solution $z\in H^1(\Omega\setminus\Sigma)^2$
such that $z=0$ on $\Gamma^{\rm D}$, $[\![z]\!]_2\ge0\text{ on }\Sigma$,
and satisfying the VI \eqref{2.13}:
\begin{multline}\label{6.1}
\int_{\Omega\setminus\Sigma} \sigma(z)\cdot \epsilon(u -z) \,dx
+\int_{\Sigma} \bigl\{ \alpha_{\rm f}^\prime ([\![z]\!]_1)\, [\![u -z]\!]_1\\
+\alpha_{\rm c}^\prime([\![z]\!]_2)\, [\![u -z]\!]_2\bigr\} \,dS_x
\ge \int_{\Gamma^{\rm N}} g\cdot (u -z) \,dS_x
\end{multline}
for all test functions $u\in H^1(\Omega\setminus\Sigma)^2$
such that $u=0$ on $\Gamma^{\rm D}$ and $[\![u]\!]_2\ge0\text{ on }\Sigma$.
Let the observation boundary be $\Gamma^{\rm O} =\Gamma^{\rm N}$.
We insert the solution $z$ of \eqref{6.1} as a measurement into the objective function
$\mathcal{J}$ in \eqref{1.3} and consider the shape optimization problem:
find $\Sigma_t$ from the feasible set $\mathfrak{S}
=\{x\in\Omega:\, x_1\in(0, 1), x_2=\psi(x_1)\in(0, 0.5),\, \psi\in C^{0,1}(0,1)\}$ such that
\begin{equation}\label{6.2}
\min_{\Sigma_t\in\mathfrak{S}}\mathcal{J}(u_t; \Omega_t)
=\frac{1}{2} \int_{\Gamma^{\rm O}_t} |u_t - z|^2 \,dS_x +\rho |\Sigma_t|,
\quad\text{where $u_t$ solves \eqref{2.13}}.
\end{equation}
Evidently, the trivial minimum in \eqref{6.2} is attained as $\Sigma_t =\Sigma$
and $u_t=z$.
To avoid the inverse crime, we use two different meshes for $z$,
and for $u_t$ when solving the inverse problem.

Now we discretize the problem.
For fixed $t$, let $\Omega^1_{t,h}$,  $\Omega^2_{t,h}$ be triangular meshes
with grid size $h>0$ in $\Omega^1_{t}$,  $\Omega^2_{t}$,
which are compatible at the interface such that $\Sigma_{t,h}
:=\Sigma_{t}\cap\partial\Omega^1_{t,h} =\Sigma_{t}\cap\partial\Omega^2_{t,h}$.
At the interface $\Sigma_{t,h}$ the nonlinear functions are set:
friction $\alpha_{\rm f}$ from \eqref{2.4} with $F_{\rm b} =10^{-5}$ (mPa);
cohesion $\alpha_{\rm c}$ from \eqref{2.6} with $m=1$,
$K_{\rm c} =10^{-3}$ (mPa$\cdot$m), $\kappa =10^{-2}$ (m).
The parameters $\delta$, $h$ are assumed sufficiently small
such that we rely on the discretization:
\begin{multline}\label{6.3}
(\alpha_{\rm f})_h(s) =F_{\rm b} |s|,\quad (\alpha_{\rm f}^\prime)_h(s)
=F_{\rm b}\, {\rm sgn}(s),\quad (\alpha_{\rm f}^{\prime\prime})_h(s) =0;\\
(\alpha_{\rm c})_h(s) =\frac{K_{\rm c}}{\kappa} \min(\kappa, |s|),\quad
(\alpha_{\rm c}^\prime)_h(s) =\frac{K_{\rm c}}{\kappa} \, {\rm ind}\{|s|<\kappa\},
\quad (\alpha_{\rm c}^{\prime\prime})_h(s) =0.
\end{multline}
\begin{figure}[hbt!]
\begin{center}
\epsfig{file=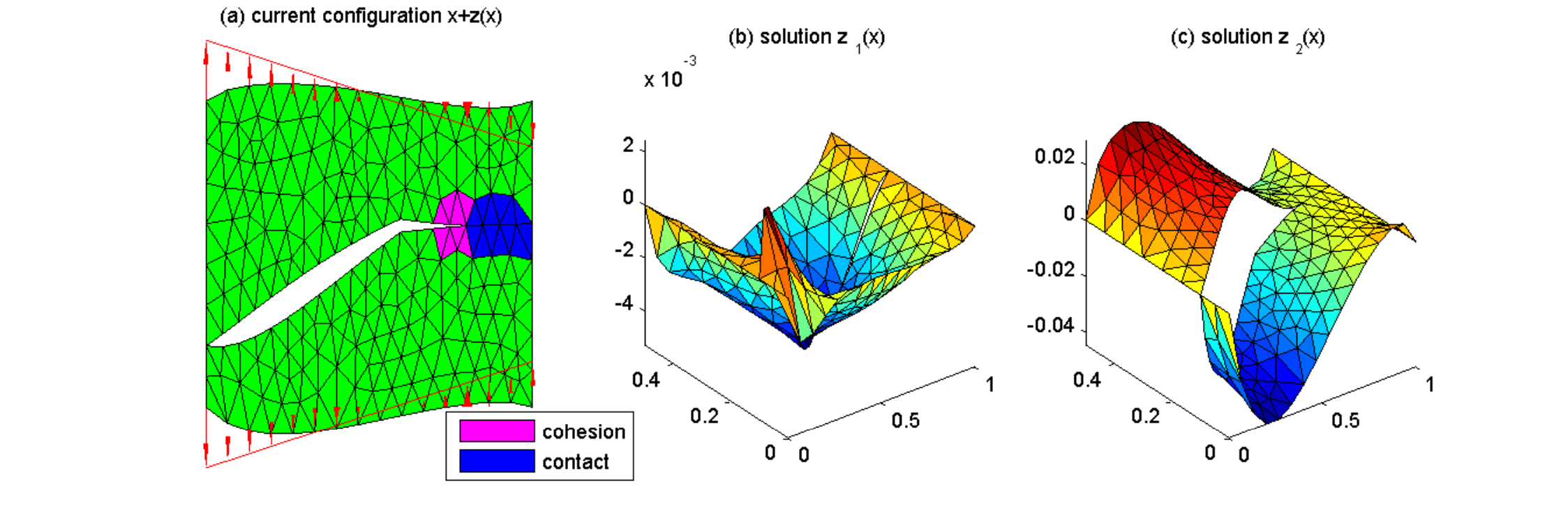,width=\textwidth,angle=0}
\caption{Computed true solution $z_h$ to \eqref{6.1}
within current configuration (a); componentwise in (b), (c).}
\label{fig_zcontact}
\end{center}
\end{figure}
After FE-discretization of problem \eqref{6.1} according to \eqref{6.3}
on a grid of size $h =10^{-2}$, we solve it by a primal-dual active set (PDAS)
iterative algorithm developed in \cite{HKK/11}.
The reference numerical solution $z_h$ obtained after 4 iterations
is plotted in Figure~\ref{fig_zcontact}.
In plot (a) we present the grid in the so-called current or deformed configuration $x+z(x)$ for
$x\in\Omega\setminus\Sigma$ under the traction force $g$ prescribed at $\Gamma^{\rm N}$.
Here we observe an open part of $\Sigma$ which is the complement to
the cohesion part (where $[\![z]\!]_2<\kappa$) with contact (where $[\![z]\!]_2=0$)
marked by colors in finite elements adjacent to the interface.
In plots (b), (c) of Figure~\ref{fig_zcontact} the solution components $(z_h)_1$, $(z_h)_2$
in the reference configuration $\Omega\setminus\Sigma$ are depicted.

According to Theorem~\ref{theo2} we approximate the VI \eqref{2.13}
by the $\varepsilon$-regularized cohesive crack problem \eqref{3.4}.
For sufficiently small $\varepsilon$ fixed, the compliance $\beta_{\varepsilon}$
from \eqref{3.3} is discretized as
\begin{equation}\label{6.4}
(\beta_{\varepsilon})_h(s) =\frac{1}{\varepsilon} \min(0, s),\quad
(\beta_{\varepsilon}^\prime)_h(s) =\frac{1}{\varepsilon} \, {\rm ind}\{s<0\}.
\end{equation}
Let $V_{t,h}$ be the finite element (FE) space of piecewise-linear functions such that
\begin{equation*}
V_{t,h}\subset V(\Omega_{t,h}) =\{u\in  H^1(\Omega^+_{t,h})^2\cap
H^1(\Omega^-_{t,h})^2\vert \quad u=0\text{ on }\Gamma^{\rm D}\}.
\end{equation*}
Then the discretization of the penalty equation \eqref{3.4} becomes:
find $u^\varepsilon_{t,h}\in V_{t,h}$ such that
\begin{multline}\label{6.5}
\int_{\Omega\setminus\Sigma_{t,h}} \sigma(u^\varepsilon_{t,h})\cdot
\epsilon(u_h) \,dx +\int_{\Sigma_{t,h}} \bigl\{ (\alpha_{\rm f}^\prime)_h
\bigl([\![u^\varepsilon_{t,h}]\!]_1\bigr)\cdot [\![u_h]\!]_1\\
+[(\alpha_{\rm c}^\prime)_h +(\beta_\varepsilon)_h] \bigl([\![u^\varepsilon_{t,h}]\!]_2
\bigr)\, [\![u_h]\!]_2\bigr\} \,dS_x =\int_{\Gamma^{\rm N}} g\cdot u_h \,dS_x,
\end{multline}
and due to \eqref{6.3} the discrete adjoint equation \eqref{3.15} reads:
find $v^\varepsilon_{t,h}\in V_{t,h}$ such that
\begin{multline}\label{6.6}
\int_{\Omega\setminus\Sigma_{t,h}} \sigma(v_h)\cdot \epsilon(v^\varepsilon_{t,h}) \,dx
+\int_{\Sigma_{t,h}} \int_0^1 (\beta^\prime_\varepsilon)_h
([\![r u^\varepsilon_{t,h}]\!]_2) \,[\![v_h]\!]_2 [\![v^\varepsilon_{t,h}]\!]_2 \,dr \,dS_x\\
=\int_{\Gamma^{\rm N}} (u^\varepsilon_{t,h} -z_h)\cdot v_h \,dS_x
\end{multline}
for all test functions $u_h, v_h\in V_{t,h}$.

After solving problems \eqref{6.5} and \eqref{6.6}, since $\Gamma^{\rm D}$
and $\Gamma^{\rm N} =\Gamma^{\rm O}$ are fixed in this example,
according to Corollary~\ref{corol2} we calculate $\mathcal{D}^\varepsilon_3$
at the moving boundary $\Sigma_{t,h}$, and $\mathcal{D}_1$
at $\Sigma_{t,h}\cap\Gamma^{\rm D}$:
\begin{multline}\label{6.7}
(\mathcal{D}_1)_{t,h}
=[\![ \nabla (u^\varepsilon_{t,h})^\top \sigma(v^\varepsilon_{t,h})
+\nabla (v^\varepsilon_{t,h})^\top \sigma(u^\varepsilon_{t,h}) ]\!] \tau_t (2x_1-1),\\
(\mathcal{D}^\varepsilon_3)_{t,h} =[\![\sigma(u^\varepsilon_{t,h})\cdot
\epsilon(v^\varepsilon_{t,h})]\!]
+\varkappa_{t} \bigl( \rho -(p_{\rm f})_{t,h} -(p^\varepsilon_{\rm c})_{t,h} \bigr)
-\nu_t\cdot \bigl( (\nabla p_{\rm f})_{t,h} +(\nabla p^\varepsilon_{\rm c})_{t,h} \bigr),
\end{multline}
where $\rho =1/ \mu_{\rm L}$ is set, $(q_{\rm f})_{t,h} =(q^\varepsilon_{\rm c})_{t,h} =0$
by the virtue of \eqref{6.3}, \eqref{6.4}.
Relying on a flat shape approximation we take $\nabla\nu_t =\nabla\tau_t =0$ and
\begin{multline}\label{6.8}
(p_{\rm f})_{t,h} =(\alpha_{\rm f}^\prime)_h
([\![u^\varepsilon_{t,h}]\!]_1)\, [\![v^\varepsilon_{t,h}]\!]_1,\quad
(p^\varepsilon_{\rm c})_{t,h} =[(\alpha_{\rm c}^\prime)_h +(\beta_\varepsilon)_h]
([\![u^\varepsilon_{t,h}]\!]_2)\, [\![v^\varepsilon_{t,h}]\!]_2,\\
(\nabla p_{\rm f})_{t,h} =[\![\nabla v^\varepsilon_{t,h}]\!]^\top \tau_t\,
(\alpha_{\rm f}^\prime)_h ([\![u^\varepsilon_{t,h}]\!]_1),\\
(\nabla p^\varepsilon_{\rm c})_{t,h} =[\![\nabla v^\varepsilon_{t,h}]\!]^\top\nu_t\,
[(\alpha_{\rm c}^\prime)_h +(\beta_\varepsilon)_h] ([\![u^\varepsilon_{t,h}]\!]_2)
+[\![\nabla u^\varepsilon_{t,h}]\!]^\top \nu_t\, (\beta_\varepsilon^\prime)_h
([\![u^\varepsilon_{t,h}]\!]_2)\, [\![v^\varepsilon_{t,h}]\!]_2.
\end{multline}

The discrete velocity $\Lambda_{H}$ at interface $\Sigma_{t}$ is defined on
a coarse grid of size $H>0$.
According to Corollary~\ref{corol3} we get a descent direction by setting
$(\Lambda_{H})_1=0$ and
\begin{equation}\label{6.9}
(\Lambda_{H})_2 =\frac{k}{\sqrt{h}} (2x_1-1)\nu_t \cdot (\mathcal{D}_1)_{t,h}
\text{ at $\Sigma_{t,h}\cap\Gamma^{\rm D}$},\quad
(\Lambda_{H})_2 =-k (\mathcal{D}^\varepsilon_3)_{t,h}
\text{ at $\Sigma_{t,h}\setminus\Gamma^{\rm D}$},
\end{equation}
where the scaling $k = 0.1 h/ \|(\Lambda_{H})_2\|_{C(\overline{\Sigma_{t,h}})}$
is chosen, and the weight $1/\sqrt{h}$ at $\Gamma^{\rm D}$
was found empirically as in \cite{GKK/20}.
Based on formulas \eqref{6.7}--\eqref{6.9} we formulate the shape optimization
algorithm of breaking line identification for the discretized version of \eqref{3.8} .

\begin{algo}[breaking line identification]\label{algo1}
\parbox{0cm}{}
\begin{itemize}
\item[\bf(0)]
Initialize the constant grid function $\psi^{(0)}_H =0.25$ at points $s_H\in[0,1]$.
Determine the line segment
$\Sigma^{(0)} = \{x_1\in(0,1),\, x_2 =\psi^{(0)}(x_1)\}$, where
$\psi^{(0)}$ is the linear interpolate of $\psi^{(0)}_H$; set $n=0$.\\
\item[\bf(1)]
Set the interface $\Sigma_{t,h} =\Sigma^{(n)}$ and construct triangulations
$\Omega^1_{t,h}$, $\Omega^2_{t,h}$;
find solutions $u^\varepsilon_{t,h}$, $v^\varepsilon_{t,h}$
of the discrete penalty and adjoint equations \eqref{6.5}, \eqref{6.6}.\\
\item[\bf(2)]
Calculate a velocity $(\Lambda_{H})_2$ by formula \eqref{6.9};
update the grid function
\begin{equation}\label{6.10}
\psi^{(n+1)}_H =\psi^{(n)}_H + (\Lambda_{H})_2
\quad\text{at points $s_H\in[0,1]$}.
\end{equation}
From linear interpolation $\psi^{(n+1)}$ of $\psi^{(n+1)}_H$ determine the piecewise-linear segment
\begin{equation}\label{6.11}
\Sigma^{(n+1)} = \{x_1\in(0,1),\, x_2 =\psi^{(n+1)}(x_1)\}.
\end{equation}
\item[\bf(3)]
If stopping criterion holds, then STOP; else set $n=n+1$ and go to Step~{\rm\bf(1)}.
\end{itemize}
\end{algo}

For 11 equidistant points $s_H$ as $H=0.1$, the numerical result of Algorithm~\ref{algo1}
after $\#n=200$ iterations (the stopping criterion) is depicted in Figure~\ref{fig_shapecontact}.
\begin{figure}[hbt!]
\begin{center}
\epsfig{file=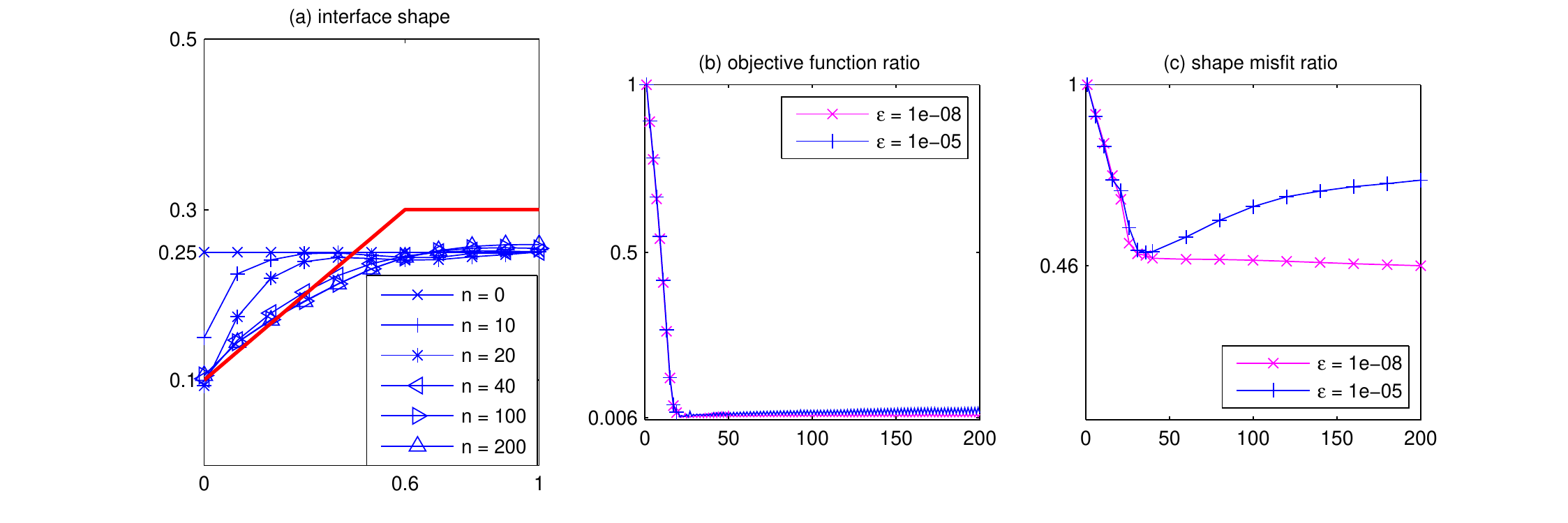,width=\textwidth,angle=0}
\caption{Iterations of $\Sigma^{(n)}$ (a);
objective function ratio $\mathcal{J}^{(n)}/\mathcal{J}^{(0)}$ (b);
shape error ratio (c).}
\label{fig_shapecontact}
\end{center}
\end{figure}
In plot (a) the selected iterations $n=0,10,20,40,100,200$ of $\Sigma^{(n)}$
from \eqref{6.11} are drawn in $\Omega$ in comparison with
the true interface $\Sigma$ (the thick solid line).
In plot (b) of Figure~\ref{fig_shapecontact} we plot the ratio
$\mathcal{J}^{(n)}/\mathcal{J}^{(0)}$ of the objective function during iterations
of $\Sigma_{t,h} =\Sigma^{(n)}$, where we recall
\begin{equation}\label{6.12}
\mathcal{J}^{(n)}(u^\varepsilon_{t,h}; \Omega\setminus\Sigma^{(n)})
=\frac{1}{2} \int_{\Gamma^{\rm O}} |u^\varepsilon_{t,h} - z_h|^2 \,dS_x
+\rho |\Sigma^{(n)}|\quad\text{subject to }\eqref{6.5}.
\end{equation}
The computed ratio attains as minimum $0,6\%$.
In plot (c) of Figure~\ref{fig_shapecontact} the ratio of shape error
$\|\Sigma^{(n)} -\Sigma\|/ \|\Sigma^{(0)} -\Sigma\|$ is plotted versus $n\in[0,200]$,
where  according to \eqref{6.10}
\begin{equation}\label{6.13}
\|\Sigma^{(n)} -\Sigma\| :=\|\psi^{(n)}-\psi\|_{C([0,1])}.
\end{equation}
Here the accuracy of shape identification attains only $46\%$.
It is worth noting that the computation is presented for small penalty parameter
$\varepsilon =10^{-8}$, while insufficiently small value $\varepsilon =10^{-5}$
causes some increase of the ratio curves after reaching the minimum;
see Figure~\ref{fig_shapecontact} (b), (c).

From the simulation we conclude the following.
In Figure~\ref{fig_shapecontact} (a) it can be observed that the left part of
curve $\Sigma$, where the constraints are inactive (see Figure~\ref{fig_zcontact} (a)),
is recovered well by the identification Algorithm~\ref{algo1},
whereas the right part of interface, where either contact or cohesion occurs,
the initialized $\Sigma^{(0)}$ is almost not modified during the iteration.

To remedy the hidden part, we apply to the same physical and geometrical configuration
the traction force $g_2(x) = (1-5 x_1/4)(4 x_2-1) \mu_{\rm L}$,
which is more stretching than the one from Figure~\ref{fig_zcontact} (a).
Because of that, the whole $\Sigma$ is open, neither contact nor cohesion occur
at the interface (see Figure~\ref{fig_shapenocontact} (a)).
The corresponding result of Algorithm~\ref{algo1} is depicted in Figure~\ref{fig_shapenocontact}.
\begin{figure}[hbt!]
\begin{center}
\epsfig{file=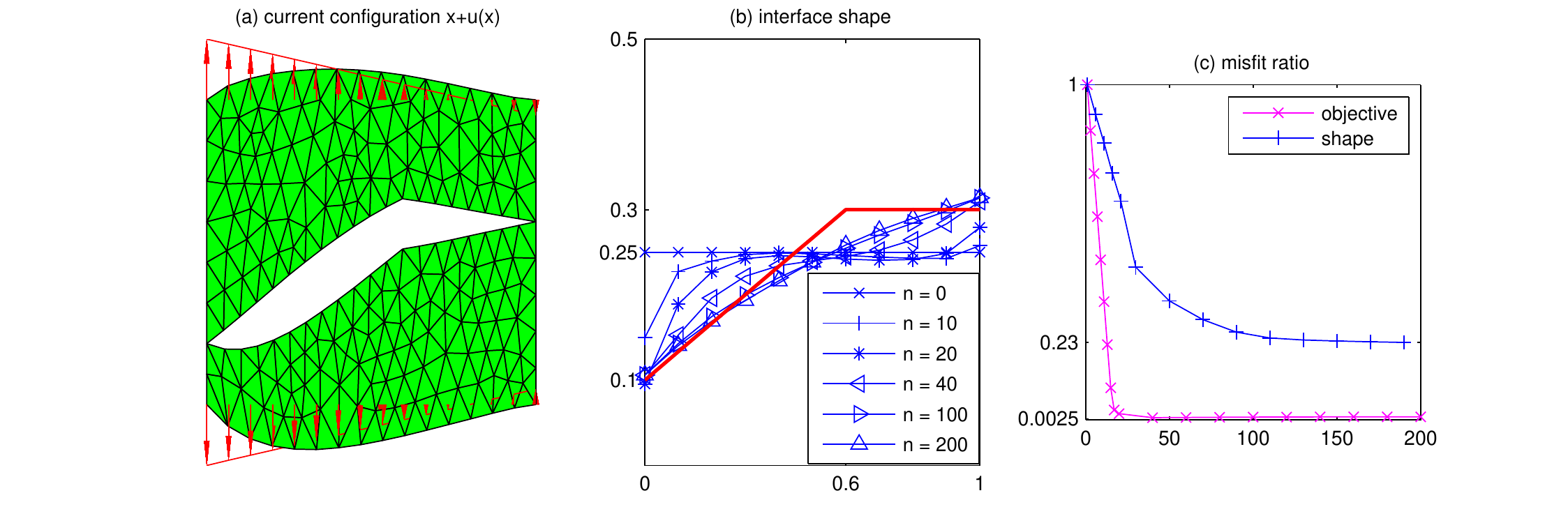,width=\textwidth,angle=0}
\caption{The true solution $z_h$ (a);
iterations of $\Sigma^{(n)}$ (b); objective function ratio and shape error ratio (c).}
\label{fig_shapenocontact}
\end{center}
\end{figure}
Here plot (b) presents the selected iterations of $\Sigma^{(n)}$, and
plot (c) shows the objective function ratio $\mathcal{J}^{(n)}/\mathcal{J}^{(0)}$
together with the shape error ratio $\|\Sigma^{(n)} -\Sigma\|/ \|\Sigma^{(0)} -\Sigma\|$.
The former ratio attains the minimum $0,25\%$, and the latter one $23\%$ of accuracy.
Now we see in Figure~\ref{fig_shapenocontact} (b) that the whole curve $\Sigma$
is recovered well compared to that from Figure~\ref{fig_shapecontact} (a).

\section{Conclusions}\label{sec7}

The Barenblatt's crack model assuming cohesion at a breaking line
is stated as the variational inequality due to the non-penetration condition
and penalized using smooth Lavrentiev's approximation.
For the geometry-dependent least-square function describing misfit
of the solution from a boundary measurement,
the expression of shape derivative is derived in an analytical form.
On its basis, from our numerical simulation we make a conclusion that
the suggested breaking line identification algorithm is consistent
within the setup of destructive physical analysis (DPA).

\paragraph{\small\bf Data availability statement}
{\small Data sharing not applicable to this article as no datasets were generated or analysed during the current study.}

\begin{acknowledgements}
The research was supported by the ERC advanced grant 668998 (OCLOC) under the EU’s H2020 research program.
\end{acknowledgements}

\appendix
\section{Proof of Lemma~\ref{lem2}}\label{A}

As $s\to0$, the following asymptotic expansion of terms in \eqref{4.11}--\eqref{4.13}
holds (see e.g. \cite[Chapter~2]{SZ/92}):
\begin{multline}\label{A1}
z\circ \phi_s =z +s \nabla z \Lambda +{\rm o}(s),\quad
g\circ \phi_s =g +s \nabla g \Lambda +{\rm o}(s),\quad
C\circ \phi_s =g +s\nabla C \Lambda +{\rm o}(s),\\
\nabla\phi_s^{-1}\circ \phi_s =I -s\nabla \Lambda +{\rm o}(s),\quad
E(\nabla\phi_s^{-1}\circ \phi_s, \tilde{u}) =\epsilon(\tilde{u})
-s E(\nabla \Lambda, \tilde{u}) +{\rm o}(s),\\
\omega^{\rm d}_s =1 +s\, {\rm div} \Lambda +{\rm o}(s),\quad
\omega^{\rm b}_s =1 +s\, {\rm div}_{\tau_t} \Lambda +{\rm o}(s)\\
\nu_{t+s}\circ \phi_s =\nu_{t} +s \nabla \nu_t \Lambda +{\rm o}(s),\quad
[\![\tilde{u}]\!]_{\tilde{\tau}_{t+s}} =[\![\tilde{u}]\!]_{\tau_t}
+s [\![\tilde{u}]\!]_{\nabla \tau_t \Lambda} +{\rm o}(s)
\end{multline}
for $\tilde{u}\in V(\Omega_{t})$.
It is worth noting that $\nabla \nu_t \Lambda$ and
$\nabla \tau_t \Lambda$ from \eqref{4.17} are just a notation used for short,
which does not require existence of the gradients here.
The tangential divergence ${\rm div}_{\tau_t} \Lambda$ is defined in \eqref{4.18}.

Inserting representations \eqref{A1} into the objective
$\tilde{\mathcal{J}}(s, \tilde{u}; \Omega_{t})$ and the perturbed Lagrangian
$\tilde{\mathcal{L}}^\varepsilon (s, u^\varepsilon_t, \tilde{u}, \tilde{v}; \Omega_{t})$
given by \eqref{4.9}, \eqref{4.10}, we derive their expansions
\eqref{4.14}, \eqref{4.15} with respect to $s$.
The asymptotic term ${\textstyle\frac{\partial}{\partial s}}
\tilde{\mathcal{L}}^\varepsilon(0,  u^\varepsilon_t, \tilde{u}, \tilde{v}; \Omega_{t})$
is from \eqref{4.16} at $\tau=0$ (implying that $\Lambda|_t =\Lambda$).
Since $\Lambda|_{t+\tau}$ and $\nabla\Lambda|_{t+\tau}$ are continuous functions
of the argument $t+\tau$, the partial derivative
$\tau\mapsto {\textstyle\frac{\partial}{\partial s}}
\tilde{\mathcal{L}}^\varepsilon (\tau,\,\cdot\,)$ in \eqref{4.16} is continuous.
This finishes the proof.

\section{Proof of Lemma~\ref{lem1}}\label{D}

The first inequality in \eqref{4.8} implies the optimality condition
$\partial_v \tilde{\mathcal{L}}^\varepsilon(s, u^\varepsilon_t,
\tilde{u}^\varepsilon_{t+s}, \tilde{v}^\varepsilon_{t+s}; \Omega_{t}) =0$, that is
\begin{multline}\label{B1}
\int_{\Omega\setminus\Sigma_{t}} \bigl( (C\circ \phi_s) E(\nabla\phi_s^{-1}\circ \phi_s,
\tilde{u}^\varepsilon_{t+s})\cdot E(\nabla\phi_s^{-1}\circ \phi_s, \tilde{v}) \bigr) \,\omega^{\rm d}_s dx\\
+\int_{\Sigma_t} \Bigl\{ \Bigl( \int_0^1 \nabla^2\alpha_{\rm f}
([\![r u^\varepsilon_t]\!]_{\tau_t}) [\![\tilde{u}^\varepsilon_{t+s}]\!]_{\tilde{\tau}_{t+s}} \,dr
+\nabla\alpha_{\rm f}(0) \Bigr) \cdot [\![\tilde{v}]\!]_{\tilde{\tau}_{t+s}}
+\Bigl( \int_0^1[\alpha_{\rm c}^{\prime\prime} +\beta^\prime_\varepsilon]
(\nu_t\cdot [\![r u^\varepsilon_t]\!]) (\tilde{\nu}_{t+s}
\cdot[\![\tilde{u}^\varepsilon_{t+s}]\!]) \,dr\\
+[\alpha_{\rm c}^\prime +\beta_\varepsilon](0) \Bigr)
(\tilde{\nu}_{t+s} \cdot[\![\tilde{v}]\!]) \Bigr\} \,\omega^{\rm b}_s dS_x
=\int_{\Gamma^{\rm N}_{t}}(g\circ \phi_s)\cdot \tilde{v} \,\omega^{\rm b}_s dS_x
\quad\text{for all } \tilde{v}\in V(\Omega_{t}).
\end{multline}
According to the asymptotic representation \eqref{4.15} and the mean value theorem,
using the operator $A_\varepsilon$ from \eqref{3.15}
it is possible to express the equation \eqref{B1} in the form
\begin{multline}\label{D1}
\langle A_{\varepsilon}(u^\varepsilon_t) \tilde{u}^\varepsilon_{t+s}, \tilde{v} \rangle
+\int_{\Sigma_t} \bigl( \nabla \alpha_{\rm f}(0)\cdot [\![\tilde{v}]\!]_{\tau_t}
+[\alpha_{\rm c}^\prime +\beta_\varepsilon](0) (\nu_t\cdot[\![\tilde{v}]\!]) \bigr) \,dS_x\\
=\int_{\Gamma^{\rm N}_t} g\cdot \tilde{v} \,dS_x +s R_v(\alpha^v_s,
\tilde{u}^\varepsilon_{t+s}, \tilde{v})
\quad\text{for all $\tilde{v}\in V(\Omega_{t})$},\quad \alpha^v_s\in(0, s),
\end{multline}
with a bounded, bilinear residual $R_v: V(\Omega_{t})^2\mapsto\mathbb{R}$.
Under assumption \eqref{3.14} the operator $A_{\varepsilon}(u^\varepsilon_t)$ 
is coercive (see \eqref{3.17}) and weakly continuous.
Thus by the Brouwer fixed point theorem, for small $s$ the variational equation \eqref{D1}
has a unique solution $\tilde{u}^\varepsilon_{t+s}\in V(\Omega_{t})$.

Similarly, the optimality condition
$\partial_u \tilde{\mathcal{L}}^\varepsilon(s, u^\varepsilon_t,
\tilde{u}^\varepsilon_{t+s}, \tilde{v}^\varepsilon_{t+s}; \Omega_{t}) =0$ reads as
\begin{multline}\label{B4}
\int_{\Omega\setminus\Sigma_{t}} \bigl( (C\circ \phi_s)
E(\nabla\phi_s^{-1}\circ \phi_s, \tilde{u})\cdot E(\nabla\phi_s^{-1}\circ \phi_s,
\tilde{v}^\varepsilon_{t+s}) \bigr) \,\omega^{\rm d}_s dx\\
+\int_{\Sigma_t} \int_0^1 \Bigl\{ \Bigl( \nabla^2 \alpha_{\rm f}
([\![r u^\varepsilon_t]\!]_{\tau_t}) [\![\tilde{u}]\!]_{\tilde{\tau}_{t+s}} \Bigr)
\cdot [\![\tilde{v}^\varepsilon_{t+s}]\!]_{\tilde{\tau}_{t+s}}
+[\alpha_{\rm c}^{\prime\prime} +\beta^\prime_\varepsilon]
(\nu_t\cdot[\![r u^\varepsilon_t]\!]) \, (\tilde{\nu}_{t+s}\cdot[\![\tilde{u}]\!])
(\tilde{\nu}_{t+s} \cdot [\![\tilde{v}^\varepsilon_{t+s}]\!]) \Bigr\} \,\omega^{\rm b}_s  dr\, dS_x\\
=\int_{\Gamma^{\rm O}_{t}} (\tilde{u}^\varepsilon_{t+s} -z\circ \phi_s)
\cdot \tilde{u} \,\omega^{\rm b}_s dS_x\quad\text{for all $\tilde{u}\in V(\Omega_{t})$}.
\end{multline}
The second inequality in \eqref{4.8} admits the decomposition
for a weight $\alpha^u_s\in(0, s)$:
\begin{equation}\label{D2}
\langle A_{\varepsilon}(u^\varepsilon_t) \tilde{u}, \tilde{v}^\varepsilon_{t+s} \rangle
=\int_{\Gamma^{\rm O}_t} (\tilde{u}^\varepsilon_{t+s} -z)\cdot \tilde{u} \,dS_x
+s R_u(\alpha^u_s, \tilde{v}^\varepsilon_{t+s}, \tilde{u})
\quad\text{for all $\tilde{u}\in V(\Omega_{t})$},
\end{equation}
with bounded bilinear $R_u: V(\Omega_{t})^2\mapsto\mathbb{R}$,
thus possesses a unique solution $\tilde{v}^\varepsilon_{t+s}\in V(\Omega_{t})$,
for $s$ small enough.

\section{Proof of Lemma~\ref{lem3}}\label{B}

\paragraph{\small\bf Uniform estimate of $\tilde{u}^\varepsilon_{t+s}$.}

Testing the variational equation \eqref{B1} with $\tilde{v} =\tilde{u}^\varepsilon_{t+s}$
and applying the asymptotic expansion \eqref{D1} it follows
\begin{multline}\label{B2}
\int_{\Omega\setminus\Sigma_{t}} \sigma(\tilde{u}^\varepsilon_{t+s})\cdot
\epsilon(\tilde{u}^\varepsilon_{t+s}) \,dx
+\int_{\Sigma_t} \Bigl\{ \Bigl( \int_0^1 \nabla^2\alpha_{\rm f}
([\![r u^\varepsilon_t]\!]_{\tau_t}) [\![\tilde{u}^\varepsilon_{t+s}]\!]_{\tau_t} \,dr
+\nabla\alpha_{\rm f}(0) \Bigr) \cdot [\![\tilde{u}_{t+s}]\!]_{\tau_t}\\
+\Bigl( \int_0^1 [\alpha_{\rm c}^{\prime\prime} +\beta^\prime_\varepsilon]
(\nu_t\cdot [\![r u^\varepsilon_t]\!]) (\nu_t \cdot[\![\tilde{u}^\varepsilon_{t+s}]\!]) \,dr
+[\alpha_{\rm c}^\prime +\beta_\varepsilon](0) \Bigr) (\nu_t \cdot[\![\tilde{u}_{t+s}]\!])
\Bigr\} \,dS_x\\
=\int_{\Gamma^{\rm N}_{t}} g\cdot \tilde{u}^\varepsilon_{t+s} \,dS_x
+s R_v(\alpha^v_s, \tilde{u}^\varepsilon_{t+s}, \tilde{u}^\varepsilon_{t+s}).
\end{multline}
We apply to \eqref{B2}  the Cauchy--Schwarz, Korn--Poincare \eqref{2.15}
and trace inequalities \eqref{2.17}.
By the virtue of boundedness of $\nabla \alpha_{\rm f}$, $\nabla^2 \alpha_{\rm f}$,
$\alpha_{\rm c}^\prime$, $\alpha_{\rm c}^{\prime\prime}$, $\beta_\varepsilon$
and $\beta^\prime_\varepsilon\ge0$ in \eqref{2.3}, \eqref{2.5}, \eqref{3.1},
we derive the estimate:
\begin{equation}\label{B3}
(K_{\rm f c 2} -C_1 |s|) \|\tilde{u}^\varepsilon_{t+s}\|_{H^1(\Omega\setminus\Sigma_t)^d}
\le \sqrt{2} K_{\rm tr} \bigl( \|g\|_{L^2(\Gamma^{\rm N}_t)^d}
+(K_{\rm f 1} +K_{\rm c 1} -\beta_\epsilon(0)) \sqrt{|\Sigma_t|} \bigr) +C_1 |s|,
\quad C_1>0,
\end{equation}
uniform in $\varepsilon$ and $s\le s_0$ for sufficiently small $s_0>0$,
where $K_{\rm f c 2} := K_{\rm KP} -(K_{\rm f 2} +K_{\rm c 2}) 2 K_{\rm tr}^2 >0$
due to the assumption \eqref{3.14}.

\paragraph{\small\bf Uniform estimate of $\tilde{v}^\varepsilon_{t+s}$.}

We test the variational equation \eqref{B4} with $\tilde{u} =\tilde{v}^\varepsilon_{t+s}$.
and apply \eqref{D2}:
\begin{multline}\label{B5}
\int_{\Omega\setminus\Sigma_{t}} \sigma(\tilde{v}^\varepsilon_{t+s})\cdot
\epsilon(\tilde{v}^\varepsilon_{t+s}) \,dx +\int_{\Sigma_t} \int_0^1
\Bigl\{ \Bigl( \nabla^2 \alpha_{\rm f} ([\![r u^\varepsilon_t]\!]_{\tau_t})
[\![\tilde{v}^\varepsilon_{t+s}]\!]_{\tau_t} \Bigr)
\cdot [\![\tilde{v}^\varepsilon_{t+s}]\!]_{\tau_t}\\
+[\alpha_{\rm c}^{\prime\prime} +\beta^\prime_\varepsilon]
(\nu_t\cdot[\![r u^\varepsilon_t]\!])
(\nu_t\cdot[\![\tilde{v}^\varepsilon_{t+s}]\!])^2 \Bigr\} \,dr \,dS_x
=\int_{\Gamma^{\rm O}_{t}} (\tilde{u}^\varepsilon_{t+s} -z)\cdot
\tilde{v}^\varepsilon_{t+s} \,dS_x
+s R_u(\alpha^u_s, \tilde{v}^\varepsilon_{t+s}, \tilde{v}^\varepsilon_{t+s}).
\end{multline}
With the help of Cauchy--Schwarz, Korn--Poincare and trace inequalities \eqref{2.15},
\eqref{2.17}, due to the bondedness of $\nabla^2 \alpha_{\rm f}$,
$\alpha_{\rm c}^{\prime\prime}$, $\beta^\prime_\varepsilon\ge0$ in \eqref{2.3},
\eqref{2.5}, \eqref{3.1}, from \eqref{B5} we derive the uniform estimate:
there exists $C_2>0$ such that
\begin{equation}\label{B6}
(K_{\rm f c 2} -C_2 |s|) \|\tilde{v}^\varepsilon_{t+s}\|_{H^1(\Omega\setminus\Sigma_t)^d}
\le \sqrt{2} K_{\rm tr} \|\tilde{u}^\varepsilon_{t+s} -z\|_{L^2(\Gamma^{\rm O}_t)^d}
+C_2 |s|.
\end{equation}
Thus, for small $|s|< K_{\rm f c 2}/\min(C_1, C_2)$
relations \eqref{B3} and \eqref{B6} together give
\begin{equation}\label{B7}
\|\tilde{u}^\varepsilon_{t+s}\|_{H^1(\Omega\setminus\Sigma_t)^d}
+\|\tilde{v}^\varepsilon_{t+s}\|_{H^1(\Omega\setminus\Sigma_t)^d}
\le K,\quad K\ge0.
\end{equation}

\paragraph{\small\bf Weak convergence of
$(\tilde{u}^\varepsilon_{t+s}, \tilde{v}^\varepsilon_{t+s})$.}

By the virtue of the uniform estimate \eqref{B7}, there exists
a subsequence $s_k\to0$ as $k\to\infty$, and a weak accumulation point
$(\tilde{u}^\varepsilon_t, \tilde{v}^\varepsilon_t)\in V(\Omega_{t})^2$ such that
\begin{equation}\label{B8}
(\tilde{u}^\varepsilon_{t+s_k}, \tilde{v}^\varepsilon_{t+s_k})\rightharpoonup
(\tilde{u}^\varepsilon_t, \tilde{v}^\varepsilon_t)\quad\text{weakly in
$H^1(\Omega\setminus\Sigma_t)^{2d}, H^{1/2}(\partial\Omega^\pm_t)^{2d}$ as $s_k\to0$}.
\end{equation}
By the compactness of embedding of the boundary traces it follows that
\begin{equation}\label{B9}
(\tilde{u}^\varepsilon_{t+s_k}, \tilde{v}^\varepsilon_{t+s_k})\to
(\tilde{u}^\varepsilon_t, \tilde{v}^\varepsilon_t)\quad
\text{strongly in $L^2(\partial\Omega^\pm_t)^{2d}$ as $s_k\to0$}.
\end{equation}
Next we take the limit in \eqref{B1} and \eqref{B4} with $s=s_k$  as $k\to \infty$. Due to the uniform
continuity of $\nabla \alpha_{\rm f}$, $\alpha_{\rm c}^\prime, \beta_\varepsilon$
and $\nabla^2 \alpha_{\rm f}$, $\alpha_{\rm c}^{\prime\prime}, \beta^\prime_\varepsilon$, and
using \eqref{3.9} we arrive at the variational equations \eqref{3.4} and \eqref{3.15}, respectively.
Therefore, $(\tilde{u}^\varepsilon_t, \tilde{v}^\varepsilon_t)
=(u^\varepsilon_t, v^\varepsilon_t)$.

\paragraph{\small\bf Strong convergence of $\tilde{u}^\varepsilon_{t+s}$.}

With the help of asymptotic relation \eqref{B2} and equation \eqref{3.4}
with $u =u^\varepsilon_t$, using the Korn--Poincare inequality \eqref{2.15},
we rearrange the terms as follows
\begin{multline}\label{B10}
K_{\rm KP} \|\tilde{u}^\varepsilon_{t+s} -u^\varepsilon_t
\|^2_{H^1(\Omega\setminus\Sigma_t)^d}
\le \int_{\Omega\setminus\Sigma_{t}} \sigma(\tilde{u}^\varepsilon_{t+s}
-u^\varepsilon_t)\cdot \varepsilon(\tilde{u}^\varepsilon_{t+s} -u^\varepsilon_t) \,dx\\
=\int_{\Omega\setminus\Sigma_{t}} \bigl\{ \sigma(\tilde{u}^\varepsilon_{t+s})
\cdot \varepsilon(\tilde{u}^\varepsilon_{t+s})
-\sigma(u^\varepsilon_t)\cdot \varepsilon(u^\varepsilon_t)
-2\sigma(\tilde{u}^\varepsilon_{t+s} -u^\varepsilon_t)\cdot
\varepsilon(u^\varepsilon_t)\bigr\} \,dx
=\int_{\Gamma^{\rm N}_{t}} g\cdot (\tilde{u}^\varepsilon_{t+s} -u^\varepsilon_t) \,dS_x\\
-2 \int_{\Omega\setminus\Sigma_{t}} \sigma(\tilde{u}^\varepsilon_{t+s}
-u^\varepsilon_t)\cdot \varepsilon(u^\varepsilon_t) \,dx
-\int_{\Sigma_t} \bigl\{ \bigl( \nabla\alpha_{\rm f}
([\![\tilde{u}^\varepsilon_{t+s}]\!]_{\tau_t}) \cdot
[\![\tilde{u}^\varepsilon_{t+s}]\!]_{\tau_t} -\nabla\alpha_{\rm f}
([\![u^\varepsilon_t]\!]_{\tau_t}) \cdot [\![u^\varepsilon_t]\!]_{\tau_t} \bigr)\\
+\bigl( [\alpha_{\rm c}^\prime +\beta_\varepsilon] (\nu_t \cdot
[\![\tilde{u}^\varepsilon_{t+s}]\!]) (\nu_t \cdot[\![\tilde{u}^\varepsilon_{t+s}]\!]
-[\alpha_{\rm c}^\prime +\beta_\varepsilon] (\nu_t \cdot [\![u^\varepsilon_t]\!])
(\nu_t \cdot[\![u^\varepsilon_t]\!] \bigr) \bigr\} \,dS_x +{\rm O}(|s|).
\end{multline}
Taking the limit in \eqref{B10} as $s_k\to0$, due to the convergence established in
\eqref{B8} and \eqref{B9},  we conclude that
\begin{equation}\label{B11}
\|\tilde{u}^\varepsilon_{t+s_k} -u^\varepsilon_t \|_{H^1(\Omega\setminus\Sigma_t)^d}
\to 0\quad\text{as $s_k\to0$}.
\end{equation}

\paragraph{\small\bf Strong convergence of $\tilde{v}^\varepsilon_{t+s}$.}

We subtract equation \eqref{3.15} from \eqref{B4} and use asymptotic expansions
\eqref{A1} such that
\begin{multline}\label{B12}
\int_{\Omega\setminus\Sigma_{t}} \varepsilon(\tilde{v})\cdot
\sigma(\tilde{v}^\varepsilon_{t+s} -v^\varepsilon_t)\,dx
=\int_{\Sigma_t} \int_0^1 \Bigl\{ \Bigl( \nabla^2 \alpha_{\rm f}
([\![r \tilde{u}^\varepsilon_{t+s}]\!]_{\tau_t}) [\![\tilde{v}^\varepsilon_{t+s}]\!]_{\tau_t}
-\nabla^2 \alpha_{\rm f} ([\![r u^\varepsilon_t]\!]_{\tau_t})
[\![v^\varepsilon_t]\!]_{\tau_t} \Bigr) \cdot [\![\tilde{v}]\!]_{\tau_t}\\
+\bigl( [\alpha_{\rm c}^{\prime\prime} +\beta^\prime_\varepsilon]
(\nu_t\cdot[\![r \tilde{u}^\varepsilon_{t+s}]\!])
(\nu_t\cdot[\![\tilde{v}^\varepsilon_{t+s}]\!])
-[\alpha_{\rm c}^{\prime\prime} +\beta^\prime_\varepsilon]
(\nu_t\cdot[\![r u^\varepsilon_t]\!]) (\nu_t\cdot[\![v^\varepsilon_t]\!]) \bigr)
(\nu_t\cdot[\![\tilde{v}]\!]) \Bigr\} \,dr \,dS_x +{\rm O}(|s|).
\end{multline}
Applying to \eqref{B12} the Cauchy--Schwarz inequality, due to
the properties of $\nabla^2 \alpha_{\rm f}$, $\alpha_{\rm c}^{\prime\prime}$, $\beta^\prime_\varepsilon$ in \eqref{2.3}, \eqref{2.5}, \eqref{3.1},
we obtain the upper bound
\begin{multline}\label{B13}
\int_{\Omega\setminus\Sigma_{t}} \varepsilon(\tilde{v})\cdot
\sigma(\tilde{v}^\varepsilon_{t+s} -v^\varepsilon_t)\,dx
\le K_{\rm f 2} \| [\![\tilde{v}^\varepsilon_{t+s} -v^\varepsilon_t]\!]_{\tau_t}
\|_{L^2(\Sigma_t)^d} \|[\![\tilde{v}]\!]_{\tau_t}\|_{L^2(\Sigma_t)^d}\\
+\int_0^1 \Bigl\{ \bigl\| \nabla^2 \alpha_{\rm f}
([\![r \tilde{u}^\varepsilon_{t+s}) ]\!]_{\tau_t}
-\nabla^2 \alpha_{\rm f} ([\![r u^\varepsilon_t]\!]_{\tau_t})
\bigr\|_{L^2(\Sigma_t)^{d\times d}}
\|[\![v^\varepsilon_t]\!]_{\tau_t}\|_{L^4(\Sigma_t)^d}
\|[\![\tilde{v}]\!]_{\tau_t}\|_{L^4(\Sigma_t)^d}\\
+\bigl\| [\alpha_{\rm c}^{\prime\prime} +\beta^\prime_\varepsilon]
(\nu_t\cdot [\![r \tilde{u}^\varepsilon_{t+s}) ]\!]
-[\alpha_{\rm c}^{\prime\prime} +\beta^\prime_\varepsilon]
(\nu_t\cdot [\![r u^\varepsilon_t]\!])  \bigr\|_{L^2(\Sigma_t)}
\|\nu_t\cdot [\![v^\varepsilon_t]\!]\|_{L^4(\Sigma_t)}
\|\nu_t\cdot [\![\tilde{v}]\!]\|_{L^4(\Sigma_t)} \Bigr\} \,dr\\
+\bigl(K_{\rm c 2} +\frac{K_{\beta 1}}{\varepsilon}\bigr)
\| \nu_t\cdot [\![\tilde{v}^\varepsilon_{t+s} -v^\varepsilon_t]\!]
\|_{L^2(\Sigma_t)} \| \nu_t\cdot [\![\tilde{v}]\!]\|_{L^2(\Sigma_t)}
+C|s|, \quad C>0.
\end{multline}
By the Sobolev embedding theorem the continuity property holds:
\begin{equation}\label{B14}
\|u\|_{L^4(\partial\Omega^\pm_t)^d}\le K_{\rm emb}
\|u\|_{H^{1/2}(\partial\Omega^\pm_t)^d},\quad u\in H^1(\Omega^\pm_t)^d,
\quad d=2,3.
\end{equation}
Then \eqref{B14}, Korn--Poincare and trace inequalities \eqref{2.15}, \eqref{2.17},
together with convergences \eqref{B8}, \eqref{B9} guarantee that for fixed $\varepsilon$:
\begin{equation}\label{B15}
K_{\rm KP}\|\tilde{v}^\varepsilon_{t+s_k} -v^\varepsilon_t
\|_{H^1(\Omega\setminus\Sigma_t)^d}\le \sup_{\tilde{v}\in V(\Omega_{t})}
\frac{1}{\|\tilde{v}\|_{H^1(\Omega\setminus\Sigma_t)^d}}
\int_{\Omega\setminus\Sigma_{t}} \varepsilon(\tilde{v})\cdot
\sigma(\tilde{v}^\varepsilon_{t+s_k} -v^\varepsilon_t)\,dx
\to 0\text{ as $s_k\to0$}.
\end{equation}
The proof of  Lemma~\ref{lem3} is complete.

\section{Proof of Corollary~\ref{corol1}}\label{C}

Let $(u^\varepsilon_t, v^\varepsilon_t)\in H^2(\Omega^+_t)^{2d}\cap
H^2(\Omega^-_t)^{2d}$ be a solution to \eqref{3.4} and \eqref{3.15}.
We integrate by parts the domain integral over $\Omega\setminus\Sigma_{t}$
from \eqref{4.20} at $\tau =0$ so that
\begin{multline*}
I(\Omega\setminus\Sigma_{t})
:=-\int_{\Omega^\pm_{t}} \bigl( ( {\rm div} \Lambda C +\nabla C \Lambda)
\epsilon(u^\varepsilon_t)\cdot \epsilon(v^\varepsilon_t)
-\sigma(u^\varepsilon_t)\cdot E(\nabla \Lambda, v^\varepsilon_t)
-\sigma(v^\varepsilon_t)\cdot E(\nabla \Lambda, u^\varepsilon_t) \bigr) dx\\
=-\int_{\partial\Omega^\pm_{t}} \Lambda\cdot \bigl( n^\pm_t
\sigma(u^\varepsilon_t)\cdot \epsilon(v^\varepsilon_t)
-\nabla (u^\varepsilon_t)^\top \sigma(v^\varepsilon_t) n^\pm_t
-\nabla (v^\varepsilon_t)^\top \sigma(u^\varepsilon_t) n^\pm_t \bigr)\, dS_x
=\int_{\Sigma_{t}} \Lambda\cdot \bigl( \nu_t
[\![\sigma(u^\varepsilon_t)\cdot \epsilon(v^\varepsilon_t)]\!]\\
-[\![\nabla (u^\varepsilon_t)^\top \sigma(v^\varepsilon_t)]\!] \nu_t
-[\![\nabla (v^\varepsilon_t)^\top \sigma(u^\varepsilon_t)]\!] \nu_t \bigr)\, dS_x
+\int_{\Gamma^{\rm D}_{t}\cup \Gamma^{\rm N}_{t}} \!\!
\Lambda\cdot \bigl(  \nabla (u^\varepsilon_t)^\top \sigma(v^\varepsilon_t) n_t
+\nabla (v^\varepsilon_t)^\top \sigma(u^\varepsilon_t) n_t \bigr)\, dS_x,
\end{multline*}
where we use the assumption $n_t\cdot \Lambda =0$ at $\partial\Omega$.
Using boundary conditions from \eqref{3.5}, \eqref{3.16}
and the notation $\mathcal{D}_1$ from \eqref{4.24} it follows that
\begin{multline}\label{C1}
I(\Omega\setminus\Sigma_{t}) =\int_{\Sigma_{t}} \Lambda\cdot \Bigl( \nu_t
[\![\sigma(u^\varepsilon_t)\cdot \epsilon(v^\varepsilon_t)]\!]
-[\![\nabla v^\varepsilon_t]\!]^\top \bigl(
\nabla\alpha_{\rm f}([\![u^\varepsilon_t]\!]_{\tau_t}) +[\alpha_{\rm c}^\prime
+\beta_\varepsilon] (\nu_t\cdot[\![u^\varepsilon_t]\!])\, \nu_t \bigr)\\
-[\![\nabla u^\varepsilon_t]\!]^\top \int_0^1 \bigl( \nabla^2 \alpha_{\rm f}
([\![r u^\varepsilon_t]\!]_{\tau_t})  \, [\![v^\varepsilon_t]\!]_{\tau_t}
+[\alpha_{\rm c}^{\prime\prime} +\beta^\prime_\varepsilon]
(\nu_t\cdot[\![r u^\varepsilon_t]\!]) \,(\nu_t\cdot[\![v^\varepsilon_t]\!])
\,\nu_t \bigr) \,dr \Bigr) \, dS_x\\
+\int_{\Gamma^{\rm O}_{t}} \Lambda\cdot \bigl(
\nabla (u^\varepsilon_t)^\top( u^\varepsilon_t -z) \bigr) \,dS_x
+\int_{\Gamma^{\rm N}_{t}} \Lambda\cdot (\nabla (v^\varepsilon_t)^\top g) \,dS_x
+\int_{\Gamma^{\rm D}_{t}} \Lambda\cdot \mathcal{D}_1
(u^\varepsilon_t, v^\varepsilon_t) \,dS_x.
\end{multline}
After substitution of \eqref{C1} into \eqref{4.20}, the integrand
at $\Sigma_{t}$ is gathered in the expression:
\begin{multline}\label{C2}
I_{\Sigma_{t}} :=-{\rm div}_{\tau_t} \Lambda\, \bigl\{  \nabla\alpha_{\rm f}
([\![u^\varepsilon_t]\!]_{\tau_t}) \cdot[\![v^\varepsilon_t]\!]_{\tau_t}
+[\alpha_{\rm c}^\prime +\beta_\varepsilon] (\nu_t\cdot[\![u^\varepsilon_t]\!])\,
(\nu_t\cdot[\![v^\varepsilon_t]\!])\bigr\}\\ +\Lambda\cdot \bigl\{ \nu_t
[\![\sigma(u^\varepsilon_t)\cdot \epsilon(v^\varepsilon_t)]\!]
-\bigl( [\![\nabla v^\varepsilon_t]\!]^\top -(\nu_t\cdot  [\![v^\varepsilon_t]\!])
\nabla\nu_t^\top -\nabla \nu_t^\top [\![v^\varepsilon_t]\!] \nu_t^\top \bigr)
\nabla\alpha_{\rm f} ([\![u^\varepsilon_t]\!]_{\tau_t})\\
-\bigl( [\![\nabla v^\varepsilon_t]\!]^\top \nu_{t}
+\nabla \nu_t^\top [\![v^\varepsilon_t]\!] \bigr)\, [\alpha_{\rm c}^\prime
+\beta_\varepsilon] (\nu_t\cdot[\![u^\varepsilon_t]\!])
-\bigl( [\![\nabla u^\varepsilon_t]\!]^\top \nu_{t} +\nabla\nu_t^\top [\![u^\varepsilon_t]\!]
\bigr) \int_0^1 [\alpha_{\rm c}^{\prime\prime}
+\beta^\prime_\varepsilon] (\nu_t \cdot[\![r u^\varepsilon_t]\!])
(\nu_t \cdot[\![v^\varepsilon_t]\!]) \,dr\\
-\bigl( [\![\nabla u^\varepsilon_t]\!]^\top -(\nu_t\cdot  [\![u^\varepsilon_t]\!])
\nabla\nu_t^\top -\nabla \nu_t^\top [\![u^\varepsilon_t]\!] \nu_t^\top\bigr)
\int_0^1 \nabla^2\alpha_{\rm f} ([\![r u^\varepsilon_t]\!]_{\tau_t})
[\![v^\varepsilon_t]\!]_{\tau_t} \,dr \bigr\}.
\end{multline}
In order to combine like terms, we exploit the calculus
\begin{equation}\label{C3}
\Lambda\cdot \nabla(\xi\cdot\eta) =\Lambda\cdot (\nabla\xi^\top\eta
+\nabla\eta^\top\xi) =\eta\cdot \nabla\xi\Lambda +\xi\cdot \nabla\eta\Lambda
\quad\text{for $\xi, \eta\in \mathbb{R}^d$}.
\end{equation}
With the help of \eqref{C3}, the gradient of the product due to friction term is calculated:
\begin{equation}\label{C4}
p_{\rm f}(\tilde{u}, \tilde{v}) :=\nabla\alpha_{\rm f}
([\![\tilde{u}]\!]_{\tau_t}) \cdot[\![\tilde{v}]\!]_{\tau_t},\quad
\nabla p_{\rm f} (\tilde{u}, \tilde{v}) =\nabla  ([\![\tilde{v}]\!]_{\tau_t})^\top
\nabla\alpha_{\rm f} ([\![\tilde{u}]\!]_{\tau_t})
+\nabla ( [\![\tilde{u}]\!]_{\tau_t})^\top \nabla^2 \alpha_{\rm f}
([\![\tilde{u}]\!]_{\tau_t})  \, [\![\tilde{v}]\!]_{\tau_t},
\end{equation}
where $\nabla ( [\![\tilde{u}]\!]_{\tau_t})^\top
= [\![\nabla\tilde{u}]\!]^\top -(\nu_t\cdot  [\![\tilde{u}]\!]) \nabla\nu_t^\top
-\nabla (\nu_t\cdot [\![\tilde{u}]\!]) \nu_t^\top$  at $\Sigma_{t}$ according to \eqref{2.1}.
Similarly, we compute the gradient for the cohesive term
\begin{multline}\label{C5}
p^\varepsilon_{\rm c}(\tilde{u}, \tilde{v}) :=[\alpha_{\rm c}^\prime
+\beta_\varepsilon] (\nu_t\cdot[\![\tilde{u}]\!])\, (\nu_t\cdot[\![\tilde{v}]\!]),
\quad \nabla p^\varepsilon_{\rm c} (\tilde{u}, \tilde{v})
=([\![\nabla \tilde{v}]\!]^\top \nu_{t} +\nabla \nu_t^\top [\![\tilde{v}]\!])\,
[\alpha_{\rm c}^\prime +\beta_\varepsilon] (\nu_t\cdot[\![\tilde{u}]\!])\\
+([\![\nabla \tilde{u}]\!]^\top \nu_{t} +\nabla\nu_t^\top [\![\tilde{u}]\!])\,
[\alpha_{\rm c}^{\prime\prime} +\beta^\prime_\varepsilon]
(\nu_t \cdot[\![\tilde{u}]\!])\, (\nu_t \cdot[\![\tilde{v}]\!]).
\end{multline}
By \eqref{C4} and \eqref{C5}, the integrand \eqref{C2} is expressed as
\begin{multline}\label{C6}
I_{\Sigma_{t}} =-{\rm div}_{\tau_t} \Lambda\, [p_{\rm f} +p^\varepsilon_{\rm c}]
(u^\varepsilon_t, v^\varepsilon_t) +\Lambda\cdot \bigl\{ \nu_t
[\![\sigma(u^\varepsilon_t)\cdot \epsilon(v^\varepsilon_t)]\!]
-\nabla [p_{\rm f} +p^\varepsilon_{\rm c}] (u^\varepsilon_t, v^\varepsilon_t)
-[\![\nabla v^\varepsilon_t]\!]^\top\nu_t \bigl( \nu_t\cdot
\nabla\alpha_{\rm f} ([\![u^\varepsilon_t]\!]_{\tau_t}) \bigr)\\
-[\![\nabla u^\varepsilon_t]\!]^\top\nu_t \Bigl( \nu_t\cdot \int_0^1 \nabla^2\alpha_{\rm f}
([\![r u^\varepsilon_t]\!]_{\tau_t}) [\![v^\varepsilon_t]\!]_{\tau_t} dr \Bigr)
-\nabla (\nu_t\cdot [\![u^\varepsilon_t]\!])^\top \int_0^1 \Bigl(
[\alpha^{\prime\prime}_{\rm c} +\beta^\prime_\varepsilon]
(\nu_t\cdot [\![r u^\varepsilon_t]\!])\\ -[\alpha^{\prime\prime}_{\rm c}
+\beta^\prime_\varepsilon] (\nu_t\cdot [\![u^\varepsilon_t]\!]) \Bigr)
 (\nu_t\cdot [\![v^\varepsilon_t]\!]) \,dr \bigr\}
-\nabla ([\![u^\varepsilon_t]\!]_{\tau_t})^\top \int_0^1 \bigl(
\nabla^2 \alpha_{\rm f} ([\![r u^\varepsilon_t]\!]_{\tau_t})
-\nabla^2 \alpha_{\rm f} ([\![u^\varepsilon_t]\!]_{\tau_t}) \bigr)
[\![v^\varepsilon_t]\!]_{\tau_t} \,dr.
\end{multline}
Introducing for short the notation of $q_{\rm f}, q^\varepsilon_{\rm c}$
in \eqref{4.26} which is based on \eqref{C6},
we rearrange the terms in the shape derivative in the form
\begin{multline}\label{C7}
{\textstyle\frac{\partial}{\partial s}} \tilde{\mathcal{L}}^\varepsilon
(0, u^\varepsilon_t, u^\varepsilon_t, v^\varepsilon_t; \Omega_{t})
=\frac{1}{2} \int_{\Gamma^{\rm O}_{t}} \bigl( {\rm div}_{\tau_t} \Lambda\,
|u^\varepsilon_t -z|^2 +\Lambda\cdot \nabla (|u^\varepsilon_t -z|^2) \bigr) dS_x
+\rho \int_{\Sigma_{t}} {\rm div}_{\tau_t} \Lambda \,dS_x\\
+\int_{\Sigma_t} \bigl\{ -{\rm div}_{\tau_t} \Lambda\,
[p_{\rm f} +p^\varepsilon_{\rm c}] (u^\varepsilon_t, v^\varepsilon_t)
+\Lambda\cdot \bigl( \nu_t [\![\sigma(u^\varepsilon_t)\cdot \epsilon(v^\varepsilon_t)]\!]
-[\nabla(p_{\rm f} +p^\varepsilon_{\rm c}) +q_{\rm f} +q^\varepsilon_{\rm c}]
(u^\varepsilon_t, v^\varepsilon_t) \bigr) \bigr\} \,dS_x\\
+\int_{\Gamma^{\rm N}_{t}} \bigl( {\rm div}_{\tau_t} \Lambda (g\cdot v^\varepsilon_t)
+\Lambda\cdot \nabla ( g\cdot v^\varepsilon_t) \bigr) \,dS_x
+\int_{\Gamma^{\rm D}_{t}} \Lambda\cdot \mathcal{D}_1
(u^\varepsilon_t, v^\varepsilon_t) \,dS_x.
\end{multline}
Since the tangential velocity, its tangential divergence, and the curvature are equal to
\begin{equation}\label{C8}
\Lambda_{\tau_t} =\Lambda -(n^\pm_{t}\cdot \Lambda) n^\pm_{t},\quad
{\rm div}_{\tau_t} \Lambda_{\tau_t} ={\rm div}_{\tau_t} \Lambda
-(n^\pm_{t}\cdot \Lambda) \varkappa^\pm_{t},\quad
\varkappa^\pm_{t} ={\rm div}_{\tau_t} n^\pm_{t}\text{ at }\partial\Omega^\pm_t,
\end{equation}
for smooth $p$ the integration along a boundary $\Gamma_t\subset \partial\Omega^\pm_t$
is given by the formula (see e.g. \cite[(2.125)]{SZ/92}):
\begin{equation}\label{C9}
\int_{\Gamma_{t}} ( {\rm div}_{\tau_t} \Lambda\, p +\Lambda\cdot \nabla p) \,dS_x
=\int_{\Gamma_{t}} (n_t\cdot \Lambda) ( \varkappa_{t} p +n_t\cdot \nabla p) \,dS_x
+\begin{cases}
(\tau_t\cdot\Lambda) p|_{\partial \Gamma_{t}}& \text{in 2D,}\\
{\displaystyle\int_{\partial \Gamma_{t}}} (b_t\cdot\Lambda) p \, dL_x & \text{in 3D.}
\end{cases}
\end{equation}
In \eqref{C9} $\tau_t$ is a tangential vector at $\partial \Gamma_{t}$
positively oriented to $n_t$ in 2D, 
and $b_t =\tau_t\times n_t$ is a binomial vector
within the moving frame at $\partial \Gamma_{t}$ in 3D.
Applying \eqref{C9} to \eqref{C7}, decomposing the vectors
in \eqref{4.21} into the normal and tangential components, and recalling that
$v^\varepsilon_t =0$ at $\partial\Gamma^{\rm N}_{t}\cap\Gamma^{\rm D}_{t}$,
we conclude with the assertion of Corollary~\ref{corol1}.

\end{document}